\documentclass[11pt]{amsart}
\usepackage{color, graphicx, amsmath, amsthm}
\usepackage{xcolor}

\usepackage{xfrac} 
\usepackage{nicefrac} 
\usepackage{bbm}

\newcommand\after{\circ}
\newcommand\intersection{\cap} 
\newcommand\interior{\operatorname{int}}
\newcommand\union{\cup}

\newcommand\R{\mathbb{R}}

\newcommand\N{\mathbb{N}}
\newcommand\dee{\mathrm{d}}

\newcommand\secondterm{c_1}
\newcommand\conv{\operatorname{conv}}
\newcommand\leb{\operatorname{Leb}}

\newcommand\flags{\operatorname{Flags}}
\newcommand\flip{r}
\newcommand\polytope{P}
\newcommand\volht{{\operatorname{Vol}}}

\newcommand\regular{D}
\newcommand\polygons{\mathcal{P}_2}
\newcommand\gltwo{\GL_2(\R)}
\newcommand\orthtwo{\mathrm O_2(\R)}
\newcommand\Rplus{\R_+}

\DeclareMathOperator{\GL}{GL}
\newcommand\density{F}

\newcommand\dotprod[2]{\langle #1, #2 \rangle}
\newcommand\bigdotprod[2]{\bigl\langle #1, #2 \bigr\rangle}

\newcommand\slopedDelta{\tilde\Delta_\tau}
\newcommand\newnewdelta{\Delta}
\newcommand\newdelta{\Delta}
\newcommand\tildedelta{\Delta}

\newcommand\tildeV{V}
\newcommand\newV{\tilde V}
\newcommand\diffV{V^*}
\newcommand\oneminusprod{G}

\newtheoremstyle{mesthm}
  {10pt plus 1pt minus 1pt}
  {9pt minus 6pt}
  {\slshape}
  {0.5cm}
  {\bfseries}
  {.}
  {1ex}
  {}
\theoremstyle{mesthm}
\newtheorem{theorem}{Theorem}[section]
\newtheorem{Definition}[theorem]{Definition}

\newtheorem{lemma}[theorem]{Lemma}

\newtheorem{corollary}[theorem]{Corollary}
\newtheorem{proposition}[theorem]{Proposition}
\newtheorem{conjecture}[theorem]{Conjecture}

\renewcommand{\emph}{\textsl}

\usepackage[abbrev]{amsrefs}
\BibSpec{article}{%
		+{}{\PrintAuthors}  		{author}
		+{,}{ \textit}     		{title}
		+{,}{ }             		{journal}
		+{}{ \textbf}       		{volume}
		+{}{ \parenthesize} 		{date}
		+{}{, no. } 		{number}
		+{,}{ }      	      		{conference}
		+{,}{ }      	      		{book}
		+{,}{ }            		{pages}
		+{,}{ }            	 	{note}
		+{,}{ }            	 	{status}
		+{,}{  \texttt } {eprint}
		+{.}{}              {transition}
	}

\title[Volume growth of Funk geometry and the flags of polytopes]
{Volume growth of Funk geometry\\ and the flags of polytopes}

\author{Dmitry Faifman \and Constantin Vernicos \and Cormac Walsh}

\date{\today}

\begin{document}

\begin{abstract}
We consider the Holmes--Thompson volume of balls in the Funk geometry
on the interior of a convex domain. We conjecture that for a fixed radius,
this volume is minimized when the domain is a simplex
and the ball is centered at the barycenter, or in the centrally-symmetric case,
when the domain is a Hanner polytope.
This interpolates between Mahler's conjecture and Kalai's flag conjecture.
We verify this conjecture for unconditional domains.

For polytopal Funk geometries, we study the asymptotics of the volume of balls
of large radius, and compute the two highest-order terms.
The highest depends only on the combinatorics, namely on the number
of flags. The second highest depends also on the geometry, and thus
serves as a geometric analogue of the centro-affine area for polytopes.

We then show that for any polytope, the second highest coefficient
is minimized by a unique choice of center point,
extending the notion of Santal\'o point.
Finally, we show that, in dimension two, this coefficient,
with respect to the minimal center point,
is uniquely maximized by affine images of the regular polygon.
\end{abstract}

\maketitle

\section{Introduction}

We are interested in the volume of metric balls in the Funk geometry---a
Finsler geometry with a non-reversible metric
defined on the interior of a convex body,
closely related to the Hilbert metric.
The forward  metric balls in this geometry take a particularly simple form:
they are merely scaled versions of the body itself.
We use the Holmes--Thompson definition of volume.

Our interest stems in part from a connection with a longstanding conjecture
in convex geometry.
As the radius $R$ of the ball tends to zero,
the volume is asymptotic to $R^n$ times
the \emph{Mahler volume} of the body,
with the polar body taken relative to the center of the ball.
The Mahler conjecture states that
the Mahler volume is minimized when the body is a simplex, or, if the body
is required to be centrally symmetric, when it is a Hanner polytope. 
We recall that a Hanner (or Hansen--Lima) polytope is any polytope that can 
be constructed from line segments by taking Cartesian products and 
polar bodies.
One may ask if the same bodies minimize the volume for all $R>0$.

\begin{conjecture}
\label{conj:main_conjecture}
Let $K$ be a centrally-symmetric convex body in $\R^n$, and let $R>0$. 
Then,
the volume of the (forward) ball in the Funk geometry
of radius $R$ centered at the origin
is minimized when the body is a Hanner polytope.
Explicitly,
\begin{equation*}
\volht_K\big(B_K(R)\big)
   \geq \volht_H\big(B_H(R)\big)
   = \frac {2^n} {n!} \frac{\left(\log(2e^R-1)\right)^n}{\omega_n},
\end{equation*}
where $H$ is any Hanner polytope of dimension $n$, and $\omega_n$ the volume of the $n$-dimensional Euclidean ball of unit radius.

If $K$ is not assumed to be centrally-symmetric, then 
the centered $n$-dimensional simplex minimizes $\volht_K\big(B_K(R)\big)$.
\end{conjecture}

Our first theorem shows that as the radius becomes large,
the asymptotic behaviour of the volume of the ball
is determined by the number of \emph{flags} of the polytope.
Let $P$ be a polytope of dimension $n$ endowed with its Funk geometry,
and denote again by $\volht_P(B_P(R))$ the Holmes--Thompson volume of the ball
of radius $R$ about the origin,
which we assume to be contained in the interior of $P$.
Also, denote by $\flags(P)$ the set of flags of $P$.

\begin{theorem}
\label{thm:first_term}
For any polytope $P$ of dimension $n$,
\begin{equation*}
\lim_{R\to\infty} \frac {\omega_n} {R^n} \volht_P\big(B_P(R)\big)
    = \frac {|\flags(P)|} {(n!)^2}.
\end{equation*}
\end{theorem}

That the asymptotic volume on the left-hand-side is proportional to
the number of flags can be established by adapting the proof
of~\cite{vernicos_walsh_flag_approximability_of_convex_bodies}
from Hilbert geometry,
where the same statement holds with a different constant. 
The proof there makes use of the invariance of the volume under
collineations (invertible projective transformations), which holds for any reasonable definition
of volume in the Hilbert case. For the Funk metric, however,
invariance of the volume
under collineations is a special property of
the Holmes--Thompson definition, as observed in~\cite{faifman_funk}.
The explicit constant $(n!)^{-2}$ can then be determined, for example,
by computing the case of Hanner polytopes,
which is done in section \ref{unconditional}. 

We will however take a different approach, and deduce this theorem
from a comprehensive computation of the first two terms
of the asymptotic expansion of $\volht_P(B_P(R))$,
as explained in Theorem \ref{thm:second_term}.

In light of Theorem~\ref{thm:first_term},
our conjecture would imply that, out of all centrally-symmetric polytopes
of a given dimension, Hanner polytopes have the fewest flags.
This statement is known as the \textsl{flag conjecture} and has been ascribed
in~\cite[Chapter 22]{shaping_space_book} to Kalai.
It appears as a special case of Conjecture C
of~\cite{kalai_the_number_of_face_of_centrally_symmetric_polytopes},
although this conjecture was shown to be false in general
in~\cite{sanyal_werner_ziegler_on_kalais_conjectures}.

We show that Conjecture \ref{conj:main_conjecture} is true
for \emph{unconditional bodies}, and describe all equality cases under
that assumption.
Recall that an unconditional body is one that is symmetric through
reflections in the coordinate hyperplanes.

\begin{theorem}
\label{thm:unconditional_funk_mahler}
Let $K\subset \R^n$ be an unconditional convex body, and let $H$ be any
Hanner polytope of the same dimension.
Then, 
\begin{itemize}
\item
$\volht_K(B_K(R)) \geq \volht_H(B_H(R))$, for all $R>0$;
\item equality holds if and only if $K$ is a Hanner polytope.
\end{itemize}
\end{theorem}

Letting $R$ tend to infinity and invoking Theorem~\ref{thm:first_term},
we then deduce the following.

\begin{corollary}
For an unconditional $n$-dimensional convex polytope $P$,
and a Hanner polytope $H$ of the same dimension, we have
$$|\flags(P)|\geq |\flags(H)|=2^n n! $$ 
\end{corollary}

While we lose track of the equality cases when passing to the limit,
we do obtain necessary conditions for equality as a corollary
of Theorem \ref{thm:second_term}---see Corollary~\ref{cor:flag_equality} below.
We remark that recently, the related $3^d$ conjecture of Kalai concerning the minimal total face number of a symmetric polytope was proved in the unconditional case in \cite{chambers}.

\medskip

Since the highest-order term in the asymptotics of the volume depends only
on the combinatorics of the polytope and not on its geometry,
it is interesting to look also at the second-highest order term.

Recall that polytopes have the following property:
if $F$ and $F'$ are two faces of dimension $i-1$ and $i+1$,
respectively, and $F \subset F'$, then there are exactly two
faces $G$ of dimension $i$ such that $F \subset G \subset F'$.
So, for each $i\in\{0, \dots, n-1\}$ and flag $f\in\flags(P)$,
there is a unique flag, which we denote $r_i f$, that differs
from $f$ only by having a different face of dimension $i$.
The group $\langle \flip_0, \dots, \flip_{n-1} \rangle$ generated by these maps
is known as the \emph{monodromy group} or \emph{connection group}
of the polytope $P$.
To express the second term in the asymptotics of the volume,
we will need a particular element of this group,
the \textsl{complete flip}
$\flip := \flip_{n-1} \after \dots \after \flip_0$.

\begin{theorem}
\label{thm:second_term}
The volume of the ball of radius $R$ in the Funk geometry on a polytope $P$
satisfies
\begin{equation*}
\omega_n\volht_P \big(B_P(R) \big)
    = c_0(P) R^n + \secondterm(P) R^{n-1} + o(R^{n-1}) \textsl{,}
\end{equation*}
where
\begin{equation*}
c_0(P) = \frac{|\flags(P)|}{(n!)^2} \textsl{,}
\end{equation*}
and
\begin{equation}
\label{eqn:second_highest_term_formula}
\secondterm(P)
    = \frac {n} {(n!)^2}
      \sum_{f\in\flags(P)}
      \log\left(1 - \Bigl\langle (\flip f)_{n-1}, {f_0}\Bigr\rangle\right).
\end{equation}
\end{theorem}

Here, we are identifying the facet $(\flip f)_{n-1}$ with the associated
vertex of the dual polytope.

\begin{corollary}\label{cor:flag_equality}
For an unconditional $n$-dimensional convex polytope $P$,
and a Hanner polytope $H$ of the same dimension, assume
\begin{equation*}
|\flags(P)|= |\flags(H)|.
\end{equation*}
Then for any flag $f\in\flags(P)$, we have $-f_0\in (rf)_{n-1}$.
\end{corollary} 

We remark that all centrally-symmetric $2$-level polytopes satisfy
this condition.
 
We next consider the dependence of the volume of the metric ball
on the choice of its center.
It is shown in~\cite{centro_affine_area_paper} that, for any $R>0$,
the volume $\volht_P(B_P(R))$ is minimized by a unique choice of center,
denoted $s_R(P)$, generalizing the notion of the Santal\'o point
of a convex body. Here we study the behavior of this point as $R$ tends to
infinity.
	
\begin{theorem}
\label{thm:funk_santalo_polytope_convergence}
Let $P$ be a polytope, and let $c_1(P,x)$ be the coefficient of the
second-highest order term of the growth at infinity of the Funk volume,
for balls centered at $x$.
Then, $x \mapsto c_1(P,x)$ is proper and strictly convex,
and finite in the interior of $P$, and hence attains its minimum
at a unique point $s_\infty(P)$ in the interior.

This point can be characterized as follows: $s_\infty(P)=0$ if and only if
the center of mass of the vertices of $P^\circ$ lies at the origin,
where each vertex is assigned a weight equal to the number of flags
of $P^\circ$ that contain it.

Moreover, we have
\begin{equation*}
s_\infty(P) = \lim_{R\to\infty} s_R(P).
\end{equation*}
\end{theorem}

The characterization of $s_\infty(P)$ in this theorem is analogous
to that of the classical Santal\'o point, namely that the Santal\'o point is at the origin
if and only if the center of mass of the polar body is at the origin.

In dimension two, the second term~(\ref{eqn:second_highest_term_formula})
of Theorem~\ref{thm:second_term} reduces to
\begin{equation}
\label{equ:second_term_dim_two}
\secondterm(P) := \frac12 \sum_{i,j: i \sim j}
                        \log\big(1 - \dotprod {e_i} {v_j}\big),
\end{equation}
where the $v_j$ are the vertices of the polygon,
the $e_i$ are the edges,
which correspond to the vertices of the dual polytope,
and $i\sim j$ means that $\dotprod {e_i} {v_j} \neq 1$
and $\dotprod {e_i} {v_{j+k}} = 1$, for some $k \in \{-1, 1\}$.
We consider how the functional $\secondterm$ varies as the polygon varies,
with the number of vertices being fixed.
It is of course $\gltwo$ invariant.
We show that every stationary point of $\secondterm$ is the image under
an element of $\gltwo$ of the regular polygon.

\begin{theorem}
\label{thm:regular_stationary}
Consider the space $\polygons^m$ of convex polygons with $m\geq 3$ vertices
containing the origin in their interior.
If the functional $\secondterm\colon \polygons^m \to \R$ is stationary at $P$,
then $P$ is
a linear image of the regular polygon $\regular^m$ with $m$ vertices,
centered at the origin.
Moreover, among all $P\in\polygons^m$ with $s_\infty(P)$ at the origin,
$c_1(P)$ is uniquely maximized by linear images of $\regular^m$.
\end{theorem}

This should be compared with~\cite{meyer_reisner_mahler_polygons}
(see also~\cite{alexander_fradelizi_zvavitch}), where the regular polygon
is shown to uniquely maximize the volume product among all polygons
with a fixed number of vertices. One could ask if the regular polygon
similarly maximizes $\volht_P(B_P(R))$ for general $0<R<\infty$.

\subsection{Plan of the paper}
After some preliminaries, we prove Theorem~\ref{thm:unconditional_funk_mahler}
in section~\ref{unconditional}.
In section~\ref{sec:volume_dual},
we decompose the polytope and its polar into flag simplices, and
show that the volume of a ball can be expressed as the sum
of a certain integral over all pairs
consisting of a flag simplex of the polytope and a flag simplex of the polar.
We then calculate the two highest-order terms
contributed by each of these pairings.
In section~\ref{sec:volume_contribution}, we combine this information to
prove Theorem~\ref{thm:second_term}.
In section~\ref{sec:santalo},
we prove Theorem~\ref{thm:funk_santalo_polytope_convergence} concerning
the convergence of the Funk--Santal\'o point,
and in section~\ref{sec:stationary} we study two-dimensional Funk geometries,
proving Theorem~\ref{thm:regular_stationary} and developing an exact
expression in dimension two for the growth rate of the volume at any
finite radius.
We establish a recursive formula for the growth of the volume
in the case of the simplex in section~\ref{sec:simplices}.
We conclude in section~\ref{sec:questions} with some further conjectures,
questions, and thoughts.

\section{Preliminaries}
\label{sec:preliminaries}

\subsection{Funk geometry}

Let $K$ be an open bounded convex body in $\R^n$.
Consider a point $p\in K$ 
and a vector $\vec{v}\in \R^n$. Since $K$ is convex,
there exists a unique real number $t>0$ such that
$p+t\vec{v}$ is on the boundary $\partial K$ of $K$.
We define the \textsl{Funk weak norm} of $\vec{v}$
at $p$ to be
\begin{equation*}
F_K(p,\vec{v}) := \frac{1}{t}\text{.}
\end{equation*}
This has all the properties of a norm, except that it is not necessarily
homogeneous, only positively homogeneous,
and so the associated unit ball is not necessarily symmetric.
Some call this the \emph{tautological Finsler structure} associated to $K$,
because at each point $p$ of $K$ the unit ball of the norm $F_K$
is the just the convex set $K$ itself centered at $p$.

Given any two points $p$ and $q$ in $K$, we can look at
all piecewise $C^1$ paths  $\gamma\colon [0,1]\to K$ from $p$ to $q$,
and compute their \emph{Funk length}
\begin{equation*}
l_F(\gamma) := \int_0^1 F_K\bigl(\gamma(t),\dot\gamma(t)\bigr) \, \dee t.
\end{equation*}
If we denote by $b$ the intersection of the half line $[p,q)$ with $\partial K$,
then the infimum of these lengths over all paths from $p$ to $q$
is the \textsl{Funk weak distance} and is equal to
\begin{equation*}
d_F(p,q) := \log\frac{pb}{qb}\text{.}
\end{equation*}
Here, $pb/qb$ is to be understood as the real number $\lambda$ such
that $\vec{pb}=\lambda\vec{qb}$.
The quantity $d_F$ is a weak distance in the sense that it is non-negative
for any two distinct points, it satisfies the triangle inequality,
and $d_F(x, x)=0$ for every point $x$.
However, it is not symmetric.
For simplicity, we will drop the adjective ``weak'' from now on.

The forward Funk metric ball centered
at $p\in K$ of radius $R$ is
\begin{equation*}
B_K(p,R) := \bigl\{z\in K \mid d_K(p,z)\leq R\bigr\}.
\end{equation*}
One easily finds that
\begin{equation*}
    B_K(p,R) = p + (1 - e^{-R})(K-p),
\end{equation*}
that is, the image of $K$ by the dilation of ratio $(1-e^{-R})$
with respect to $p$.
We shall denote by $\lambda$ the number $(1-e^{-R})$ in what follows
to simplify the exposition.

The Holmes--Thompson, sometimes also called symplectic, volume in Funk geometry can be defined as follows.
Given a Lebesgue measure $\mu$ on $\R^n$, there exists a unique dual
Lebesgue measure $\mu^*$ on the dual of $\R^n$
such that, for any basis $e=(e_1,\ldots,e_n)$ and its dual one ${e^*=(e_1^*,\ldots,e_n^*)}$ we get
\begin{equation*}
\mu^*(e_1^*\land\dots\land e_n^*)\cdot\mu(e_1\land\dots\land e_n)=1.
\end{equation*}
Denote by $K^p$ the polar body of $K$ with respect to $p\in K$.
Then, for any Borel set $U$ in $K$, its Holmes--Thompson measure is
 given by
\begin{equation*}
\volht_K(U) := \int_U \frac{\mu^*(K^p)}{\omega_n} \, \dee \mu(p)
\end{equation*}
where $\omega_n$ is the volume of the $n$-dimensional Euclidean unit ball.

Typically in computations, once coordinates $p=(p_1,\ldots,p_n)$ are fixed, one takes
$\dee \mu(p)=\dee \leb_n(p)=dp_1\cdots dp_n$ where $dp_i:=p_i^*$, hence
\begin{equation*}
	\volht_K(U) = \omega_n^{-1} \int_U |K^p| \, \dee p \quad\text{ with }\quad |K^p|=\int_{K^p} dp_1^*\cdots dp_n^*.
\end{equation*}

\subsection{Faces and Flags}

A \emph{face} of a convex polytope $P$ is the intersection of the polytope with
a closed halfspace such that no point of the interior of the polytope lies
on the boundary of the halfspace.
Note that the polytope itself is a face, as is the empty set.
The dimension of a face is the dimension of the affine subspace
it generates. As usual, $n-1$-dimensional faces are called 
\textsl{facets}, $0$-dimensional ones \textsl{vertices}
and $1$-dimensional ones \textsl{edges}.

The set of faces of a polytope is a partially ordered set with the inclusion
relation.
In fact, it forms a lattice, called the \emph{face lattice}.

A \textsl{flag} of $P$ is a sequence $f=(f_0,f_1,\ldots,f_n)$,
of faces such that, for all $i\in\{0,\ldots,n\}$,
the face $f_{i}$ is an $i$-dimensional face, and,
for $i\in\{0,\ldots,n-1\}$, we have $f_i \subset f_{i+1}$.
We denote the set of flags of $P$ by $\flags(P)$.

Polytopes have the following so-called \emph{diamond property}:
if $F$ and $F'$ are two faces of dimensions $i-1$ and $i+1$,
respectively, and $F \subset F'$, then there are exactly two
faces $G$ of dimension $i$ such that
$F \subset G \subset F'$.
For each $i\in\{0, \dots, n-1\}$,
define the map $\flip_i \colon \flags(P) \to \flags(P)$ in the following way.
For each flag $f$ of $P$, and $j \in \{-1, \dots, n\}$, let
\begin{align*}
(\flip_i f)_j :=
\begin{cases}
G', \quad\text{if $j = i$}, \\
f_j, \quad\text{otherwise},
\end{cases}
\end{align*}
where $G'$ is the unique face of dimension $i$ distinct from $f_i$
satisfying $f_{i-1} \subset G' \subset f_{i+1}$. 

The two flags $r_i f$ and $f$ are said to be \emph{$i$-adjacent}.
The group $\langle \flip_0, \dots, \flip_{n-1} \rangle$ generated by these maps
is known as the \emph{monodromy group} or \emph{connection group}
of the polytope $P$.

\subsection{Taking the polar with respect to different points}

We have the following relationship between the polar of a convex body
at the origin and its polar at another point.

\begin{lemma}
\label{lem:dual_with_respect_to_point}
Let $K$ be a convex body, and $x$ a point in $\interior K$.
Then,
\begin{equation*}
K^x := (K - x)^\circ
    = \Big\{ \frac{w} {1 - \dotprod {w} {x}}
          \mathrel{}\big|\mathrel{} w \in K^\circ \Big\}.
\end{equation*}
\end{lemma}

\section{Properties of the Holmes--Thompson volume in Funk geometry}
\label{sec:properties}

While the Funk geometry is only an affine invariant,
many of its aspects are more naturally considered within projective geometry,
as was observed in~\cite{faifman_funk}.
In particular, the associated Holmes--Thompson volume
exhibits two projective invariance properties.
Firstly, it is invariant under collineations.

\begin{proposition}
\label{prop:projective_invariant}
If $g\colon \R\mathbb P^n\to \R\mathbb P^n$ is a collineation,
and $K\subset\R^n$ is a convex body in an affine chart
such that $gK\subset\mathbb R^n$,
then $\volht_K(U)=\volht_{gK}(gU)$ for any Borel subset $U\subset K$.
\end{proposition}

Secondly, the Holmes--Thompson volume in Funk geometry is invariant
under projective polarity~\cite[Proposition 7.2]{faifman_funk}.
This can be stated as follows. 

\begin{proposition}
\label{prop:duality_volume}
If $K$ and $L$ are convex bodies in $\mathbb R^n$,
with $K \subset \interior (L)$,
and $p\in \interior(K)$,
then $\volht_{L}(K)=\volht_{K^p}(L^p)$.
\end{proposition}

We now establish a multiplicative property of the volume in Funk geometry.

\begin{lemma}
\label{lem:multiplicative_volume}
If $K\subset \R^k$ and $L\subset \R^l$ are convex bodies, then
\begin{equation}  
\begin{split}
(k+l)! \, \omega_{k + l} \, \volht_{K\times L}
    & \big(B_{K\times L}\bigl((a, b), r\bigr)\big) \\
    &= \Big( k! \, \omega_k \, \volht_K \big(B_K(a, r)\big) \Big)
          \Big( l! \, \omega_l \, \volht_L \big(B_L(b, r) \big) \Big) \text{.}
\end{split}
\end{equation}
\end{lemma}

\begin{proof}
For convex bodies $A\subset\R^k$ and $B\subset \R^l$,
we define the convex body
\begin{equation*}
A \times_1 B
    = \Big\{\big(tx, (1-t)y\big)
          \mid x\in A, \,\,\, y\in B, \,\,\, 0\leq t\leq 1 \Big\}
    \subset \R^{k+l}.
\end{equation*}
Assuming $0\in A$ and $0\in B$, it has Lebesgue volume
\begin{equation*}
|A \times_1 B| = \frac {k! \, l!} {(k+l)!} |A| |B|.
\end{equation*}

Write $\lambda=1-e^{-R}$, and assume for simplicity that $a = b = 0$.
It holds, for all $(x,y)$ in the interior of $K\times L$,
that $(K\times L)^{(x,y)}=K^x\times_1L^y$, and so
\begin{align*}
\omega_{k + l} \,
  \volht_{K\times L}\bigl(B_{K\times L}\bigl((a, b), r\bigr)\bigr)
    &= \int_{\lambda (K\times L)} |(K\times L)^{(x,y)}| \, \dee x \, \dee y \\
    &= \int_{\lambda K\times\lambda L} |K^x\times_1L^y| \, \dee x \, \dee y \\
    &= \frac {k! \, l!} {(k+l)!} \int_{\lambda K}|K^x| \, \dee x
              \int_{\lambda L}|L^y| \, \dee y,
\end{align*}
concluding the proof.
\end{proof}

We can now calculate the volume of Funk balls in Hanner polytopes.

\begin{lemma}
\label{lem:same_hanner}
Let $H$ be a Hanner polytope in $\R^n$ centered at the origin,
and let $R>0$ and $\lambda = 1-e^{-R}$.
Then,
\begin{equation*}
\volht_H \big(B_H(r)\big)
    = \frac {2^n} {n! \, \omega_n} \left( \log\frac{1+\lambda}{1-\lambda}\right)^n
    = \frac {2^n} {n! \, \omega_n} \big(\log(2e^R-1)\big)^n.
\end{equation*}
\end{lemma}

\begin{proof}

An elementary calculation shows that, for $Q=[-1,1]\subset\R$, we have
\begin{equation*}
\omega_1 \, \volht_Q \big(B_Q(0,R)\big)
    = 2 \log\frac{1+\lambda}{1-\lambda}.
\end{equation*}
The conclusion follows on applying Proposition~\ref{prop:duality_volume}
and Lemma~\ref{lem:multiplicative_volume},
since any Hanner polytope can be constructed from intervals by taking
products and polars.
\end{proof}

\section{The Mahler inequality for unconditional bodies}
\label{unconditional}

It is convenient to extend the Funk Holmes--Thompson volume to functions,
as was done in \cite{faifman_funk}.
For a convex function $\phi \colon V\to \R\cup\{+\infty\}$ defined on $V=\R^n$,
and $R>0$, set 
\begin{equation*}
\widetilde M_\lambda(e^{-\phi})
    = \frac {\lambda^n} {n!}
      \int_{V\times V^*}
        e^{-\phi(x) - \mathcal L\phi(\xi) + \lambda\langle x,\xi\rangle}
      \, \dee x \, \dee \xi \text{,}
\end{equation*}
where $\lambda=1-e^{-R}$, and $\mathcal L$ is the \emph{Legendre transform}
given by
\begin{equation*}
\mathcal L\phi(\xi)
    = \sup_{x\in V} \big( \langle x, \xi\rangle -\phi(x) \big) \text{.}
\end{equation*}

Clearly
$\widetilde M_\lambda(e^{-\phi})=\widetilde M_\lambda(e^{-\mathcal L \phi})$. 
Note also that $\mathbbm 1_K^\infty:=-\log \mathbbm 1_K$ is a convex function
when $K$ is convex.

The following was verified in \cite[Lemma 8.5]{faifman_funk}.

\begin{lemma}
\label{lem:body_function_volume}
For a convex body $K$ and $0<\lambda<1$, one has
\begin{equation*}
\omega_n \, \volht_K(\lambda K)
    = \widetilde M_\lambda(\mathbbm 1_K)
    = \widetilde M_\lambda(e^{-h_K}),
\end{equation*}
where $h_K$ is the support function of $K$.
\end{lemma}

\begin{proof}
The first equality is \cite[Lemma 8.5]{faifman_funk}.
Since $\mathcal L(\mathbbm 1^\infty_K)=h_K$, we get the second one.
\end{proof}

\begin{theorem}
Let $K\subset \R^n$ be an unconditional convex body,
and $H\subset \R^n$ a Hanner polytope centered at the origin.
Then, for any  $0<R<\infty$, we have 
\begin{equation*}
{\volht_K} \big(B_K(R) \big)\geq \volht_H \big(B_H(R) \big).
\end{equation*}
Equality is uniquely attained by Hanner polytopes.
\end{theorem}

\begin{proof}
Consider an unconditional convex function
$\phi \colon \R^n\to \R\cup\{+\infty\}$.

Unconditional here means again that it is invariant under reflections
in the coordinate hyperplanes.
Define, for each $j\geq 0$,
\begin{equation*}
I_{2j}(e^{-\phi})
    := \int_{\R^n\times\R^n} \langle x,\xi\rangle ^{2j} \,
           e^{-\phi(x)-\mathcal L\phi(\xi)} \, \dee x \, \dee\xi \text{.}
\end{equation*}

Due to the unconditionality of $\phi$, we have
\begin{align*}
I_{2j}(e^{-\phi})
    &= \sum_{|I|=j} \binom {2j} {2I}
           \int_{\R^n} x^{2I}e^{-\phi(x)} \, \dee x
           \int_{\R^n}\xi^{2I}e^{- \mathcal L \phi(\xi)} \, \dee \xi \\
    &= 4^n \sum_{|I|=j} \binom {2j} {2I}
           \int_{\R_+^n} x^{2I}e^{-\phi(x)} \, \dee x
           \int_{\R_+^n}\xi^{2I}e^{- \mathcal L \phi(\xi)} \, \dee \xi,
\end{align*}
where the sum is over all $n$-tuples $I = (i_1,\dots, i_n)$
of non-negative integers with $|I| : =\sum_{k=1}^n i_k = j$.

We are using here the notation
$x^{2I} := \prod_{k=1}^{n} x_k^{2i_k}$ and
\begin{equation*}
\binom {2j} {2I} := \frac {(2j)!} {(2 i_1)! \cdots (2 i_n)!}.
\end{equation*}

By the inequality of Fradelizi--Meyer~\cite[Theorem 6]{fradelizi_meyer},
we have
\begin{equation}
\label{eq:ineq_I}
\int_{\R_+^n} x^{2I} \, e^{-\phi(x)} \, \dee x
        \int_{\R_+^n} \xi^{2I} \, e^{-\mathcal L \phi(\xi)} \, \dee\xi
    \geq \prod_{k=1}^n \frac{\Gamma(2i_k+2)} {(2i_k+1)^2},
\end{equation}
where $\Gamma$ denotes Euler's Gamma function.
Equality, for any given $j$ and $I$,
occurs (non-uniquely) when $\phi(x) = \|x\|_{\ell_1^n} := \sum |x_k|$,
the support function of the cube $[-1,1]^n$. 

Denoting $\lambda=1-e^{-R}$, we have $B_K(R)=\lambda K$.
Since $K=-K$, we have by Lemma~\ref{lem:body_function_volume},
\begin{align*}
\omega_n \, \volht_K(\lambda K)
   &= \widetilde M_\lambda(\mathbbm 1_K)
    = \frac {\lambda^n} {n!}
          \int_{K\times\R^n} e^{-h_K(\xi)+\lambda\langle x, \xi\rangle}
          \, \dee x \, \dee\xi \\
    &= \frac {\lambda^n} {n!}
                   \int_{K\times\R^n} e^{-h_K(\xi)}
              {\cosh} \big(\lambda \langle x,\xi \rangle \big)
          \, \dee x \, \dee\xi \\
    &=\frac{1}{n!} \sum_{j=0}^\infty
                      \frac{1}{(2j)!}I_{2j}(e^{-h_K})\lambda^{n+2j}.   
\end{align*}
The last equality here is obtained by expanding the hyperbolic cosine as
a power series, and using that
$I_{2j}(\exp({-\mathbbm 1_K^\infty})) = I_{2j}(\exp({-h_K}))$.

Thus all summands in the series for $\volht_K(\lambda K)$ are simultaneously
minimized when $h_K$ is the support function of the cube.
That is, among all unconditional convex bodies in $\R^n$,
the cube minimizes $\volht_K(\lambda K)$.
Applying Lemma \ref{lem:same_hanner}, we see that the minimum is also attained
by any Hanner polytope.

For the equality cases, it follows from the proof that equality must hold
in equation~\eqref{eq:ineq_I} for every $I$,
in particular when $I=(0, \dots 0)$, so that

\begin{equation*}
\int_{\R_+^n} e^{-\mathbbm 1^\infty_K(x)} \, \dee x
    \int_{\R_+^n}e^{-h_K(\xi)} \, \dee \xi = 1,
\end{equation*}
that is 
\begin{equation*}
\frac{n!}{4^n}|K||K^\circ|= 1.
\end{equation*}
The minimizers of the volume product among unconditional convex bodies were shown to coincide with the Hanner polytopes by Meyer \cite{meyer_unconditional} and Reisner \cite{reisner_unconditional}. This concludes the proof.
\end{proof}

\section{A decomposition of the volume}
\label{sec:volume_dual}

In this section and the next, we determine the two highest order terms
of growth of the balls in the Funk geometry on a polytope $P$.
Our strategy will be to decompose both $P$ and its polar into flag simplices,
and to consider separately each pairing of a flag simplex
with a dual flag simplex.

To this end, for each face $F$ of $P$, we make a choice of point $p(F)$
in the relative interior of $F$. When $F$ is a vertex, this choice of course
can only be the vertex itself. When $F = P$, we choose $p(P)$ to be the origin.
Our choice of points leads to a flag decomposition of $P$:
to each flag $f$ of $P$ is associated a flag simplex, which is the convex
hull of the points $p(f_0), \dots, p(f_n)$.
The interiors of these simplices are pairwise disjoint, and their union
equals $P$.

Similarly, for each face $G$ of $P^\circ$, we make a choice of point $q(G)$
in the relative interior of $G$.
Again, this leads to a flag decomposition of $P^\circ$.

For each flag $f$ of $P$,
denote by $\Delta^f := \conv\{p(f_0), \dots, p(f_n)\}$
the associated flag simplex.
We also define, for each flag $f$ and $\tau \in (0, 1)$,
the truncated simplex $\Delta_\tau^f := (1 -\tau) \Delta^f$.
Recall here that $p(P) = 0$.
For each dual flag $g\in\flags(P^\circ)$ and $i\in\{0, \dots, n-1\}$,
let $x^g_i(\cdot) := 1 - \dotprod {q(g_i)} {\cdot}$.
The functions $x^g_i$ provide a set of coordinates on $\R^n$,
for any given dual flag $g$.
We denote by $\leb_n$ the $n$-dimensional Lebesgue measure.

\begin{lemma}
\label{lem:combine_all_the_flags}
The Holmes--Thompson volume of the ball $B_P(0, R)$ in the Funk geometry
is given, for $\tau := \exp(-R)$, by
\begin{equation*}
\volht\big(B_P(0, R)\big)
    = \frac {1} {n! \, \omega_n} \sum_{f\in\flags(P)} \sum_{g\in\flags(P^\circ)}
       \int_{\Delta_\tau^f}
       \frac {\dee x^g_0 \cdots \dee x^g_{n-1}} {x^g_0 \cdots x^g_{n-1}}\text{.}
\end{equation*}
\end{lemma}

\begin{proof}
The volume of the ball is
\begin{equation*}
\volht(B_P(0, R))
    = \frac {1} {\omega_n} \int_{B_P(0, R)} |P^y| \,\dee\!\leb_n(y),
\end{equation*}
where $P^y$ denotes the polar of $P$ with respect to the point $y$.
Observe that $B_P(0, R)$ is the union of the simplices $\Delta_\tau^f$,
and that the interiors of these simplices do not intersect.
We conclude that
\begin{equation}
\label{eqn:split_flags}
\volht\big(B_P(0, R)\big)
    = \frac {1} {\omega_n}
         \sum_{f\in\flags(P)} \int_{\Delta_\tau^f} |P^y| \,\dee\!\leb_n(y).
\end{equation}

Recall that the $n$-dimensional Lebesgue volume of a simplex with vertices
$v_0, \dots, v_{n-1}, 0$ is $|v_0 \land \dots \land v_{n-1}| / n!$.
So, we have
\begin{equation*}
|P^\circ|
    = \frac {1} {n!}
        \sum_{g\in\flags(P^\circ)}
        |q(g_0)\land\dots\land q(g_{n-1})|.
\end{equation*}
As we vary $y$ over $\interior P$, the shape of the polar $P^y$
changes according to Lemma~\ref{lem:dual_with_respect_to_point}.
Note that the combinatorics stay the same. In fact, the change of shape
is always by a projective map.
For each face $G$ of $P^\circ$ and point $y$ in $\interior P$,
let
\begin{equation*}
q^y(G) := \frac {q(G)} {1 - \dotprod {q(G)} {y}}.
\end{equation*}
Then, the $q^y(G)$ define a flag decomposition of $P^y$.
So, we may write
\begin{align}
|P^y| &= \frac {1} {n!}
            \sum_{g\in\flags(P^\circ)}
            |q^y(g_0)\land\dots\land q^y(g_{n-1})| \nonumber \\
\label{eqn:volume_of_dual_ball}
      &= \frac {1} {n!}
            \sum_{g\in\flags(P^\circ)}
            \frac {|q(g_0)\land\dots\land q(g_{n-1})|}
                  {\big(1 - \dotprod {q(g_0)} {y}\big)
                  \cdots \big(1 - \dotprod {q(g_{n-1})} {y}\big)}.
\end{align}
Changing to the coordinates $x^g_i$ changes the measure
in the following way:
\begin{equation*}
\dee x^g_0 \cdots \dee x^g_{n-1}
    = \big|q(g_0)\land\dots\land q(g_{n-1})\big| \, \dee\!\leb_n.
\end{equation*}
Combining this with~(\ref{eqn:split_flags}) and~(\ref{eqn:volume_of_dual_ball})
gives the result.
\end{proof}

So we see that the Funk volume can be computed by calculating an integral
of a relatively simple function over a truncated flag simplex with
respect to a coordinate system coming from the dual flag simplex.
This is done for every combination of flag simplex and dual flag simplex,
and the results are added.

We will see that the highest order term as $R$ tends to infinity is determined
by the flag simplices combined with their duals;
the second highest order term is determined by
these pairs as well as pairs where the flag simplex and the dual flag simplex
are nearest neighbours, in the sense that the associated flags differ by
exactly one face;
in the following sections we calculate the contributions of the relevant pairs.

Let us fix notation that will be used throughout this section. 
Recall that $f(x) = O(x)$ means that $f(x) / x$ is bounded,
whereas $f(x) = o(x)$ means that $f(x) / x$ converges to zero
as $x$ tends to infinity.

Recall that we have the following elementary integral:
for $a\in \R$ and $k \in \N$,
\begin{equation}
	\label{eqn:integrate_log_x_over_x}
	\int \frac {(\log x + a)^k} {x}  \,\dee x
	= \frac {1} {k + 1} (\log x + a)^{k + 1} + C.
\end{equation}
For each $k \in \N$, define the function $\density^k \colon \R_{>0}^n \to \R$
by
\begin{equation*}
\density^k(x) := \frac {(\log x_n - \log \tau)^k} {x_1 \cdots x_n}.
\end{equation*}
Here $\R_{>0}$ is the set of positive real numbers.
We use the convention that $\N$, the set of natural numbers, contains $0$.

\subsection{Combining a flag simplex and its own dual}

We start off by considering the case where the flag simplex and
the dual flag simplex are associated to the same flag.
These pairs will be seen to contribute to both the $R^n$ and $R^{n-1}$ terms.

Let
\begin{equation}
\label{eqn:Delta}
\Delta := \conv\{v_0, \dots, v_n\}
\end{equation}
be an $n$-dimensional simplex with vertices
\begin{align*}
v_0 &:= (0, 0, \dots, 0) \\
v_1 &:= (v_{1,1}, 0, \dots, 0) \\
    &\,\,\,\,\vdots \\
v_{n} &:= (v_{n1}, v_{n2}, \dots, v_{nn}).
\end{align*}
All non-zero coordinates are assumed positive.

Denote by $\Delta_\tau$ the set
$\Delta \backslash (\R^{n-1} \times [0,\tau))$,
in other words, the subset of $\Delta$ where the last coordinate
is greater than or equal to $\tau$.

We are interested in the growth of the integral of $\density^k$
over $\Delta_\tau$, as $\tau$ decreases.
It is actually the case $k = 0$ that we will use later,
but it is important for the inductive proof to consider more general $k$.

We will need the following two lemmas. The first shows that, by scaling the
coordinates, we may fit the simplex $\Delta$ into the ``ordered simplex''.

\begin{lemma}
\label{lem:quasi_standard}
Let $\Delta$ be a simplex of the form (\ref{eqn:Delta}).
Then, there exist positive real numbers $a_i$ such that
\begin{equation*}
0 \leq a_n x_n \leq \dots \leq a_1 x_1 \leq 1,
\qquad
\text{for all $x\in \Delta$}.
\end{equation*}
\end{lemma}

\begin{proof}
First, choose $a_1>0$ so that $a_1 v_{i1} \leq 1$, for all $i$.
Then, choose $a_2>0$ so that $a_2 v_{i2} \leq a_1 v_{i1}$, for all $i$.
Continue in this way until all the $a_i$ are chosen.
It is clear that any convex combination of the $v_i$ then
satisfies the required relations.
\end{proof}

The next lemma shows that, when we restrict our attention to a region where
one of the coordinates is bounded below,
we get a bound on the growth rate of the integral of $F^k$.

\begin{lemma}
\label{lem:only_x_axis_counts}
Let $U(R)$ be the integral of $\density^k$ over the intersection of
$\Delta_\tau$ with the region $\{x_j \geq \beta\}$,
where $j\in\{1,\dots,n-1\}$ and $\beta>0$, with $\tau = \exp(-R)$.
Then, $U(R) = O(R^{n+k-j})$.
\end{lemma}

\begin{proof}
Using Lemma~\ref{lem:quasi_standard}, we see that the domain of integration
is contained in
\begin{equation*}
[\alpha,\gamma]^{j} \times [\delta\tau, \gamma]^{n-j-1} \times [\tau, \gamma],
\end{equation*}
with $\delta$, $\alpha$, and $\gamma$ positive real numbers.
The integrand can be integrated easily over this set
with the help of \eqref{eqn:integrate_log_x_over_x},
and thus we get the following bound:
\begin{equation*}
U(R) \leq
    (\log \gamma - \log \alpha)^{j}
    (\log \gamma - \log \delta + R)^{n-j-1}
    \frac {(\log \gamma + R)^{k+1}} {k+1}.
\qedhere
\end{equation*}
\end{proof}

\begin{lemma}
\label{lem:integral_induction}
Let $\Delta$ be of the form \eqref{eqn:Delta},
and let $\Delta_\tau$ be defined as above.
For each $R > 0$ and $k\in\N$, write
\begin{equation*}
V(R) := \frac {1} {k!}
        \int_{\Delta_\tau} F^k
        \,\dee\!\leb_n,
\end{equation*}
where $\tau = \exp(-R)$.
Then, for all $k\in\N$,
\begin{equation*}
V(R) = \frac {R^{n+k}} {(n + k)!}
         + \frac {R^{{n+k}-1}} {({n+k}-1)!} \Big(
             \sum_{i=1}^n \log v_{i,i}
         -   \sum_{i=1}^{n-1} \log v_{i+1,i} \Big)
         + o(R^{{n+k}-1}).
\end{equation*}
\end{lemma}

\begin{proof}
We use induction on the dimension $n$.
When $n=1$, the integral can be evaluated exactly
using~(\ref{eqn:integrate_log_x_over_x}),
and we get that
\begin{equation*}
V(R) = \frac {(\log v_{1,1} + R)^{k+1}} {(k + 1)!}.
\end{equation*}
So, we see that the statement of the lemma is true in this case.

Assume that $n\geq 2$ and that the statement is true when the dimension
is $n - 1$. Choose $\alpha >0$ such that $\alpha<v_{i,1}$
for all $1\leq i\leq n$.
Decompose the integral into two parts, where $x_1 \geq \alpha$,
and where $x_1 < \alpha$.
Consider the simplices $\Delta'=\Delta\cap\{x_1=\alpha\}$
and $\Delta'_\tau=\Delta_\tau\cap\{x_1=\alpha\}$.
The vertices of $\Delta'$ are
\begin{align*}
v'_1 &:= (\alpha, 0, \dots, 0) \\
     &\,\,\,\,\vdots \\
v'_{n} &:= \Big(\alpha, \alpha \frac {v_{n,2}} {v_{n,1}},
                 \dots, \alpha \frac {v_{n,n}} {v_{n,1}}\Big).
\end{align*}
Let $\widehat\Delta'$ and $\widehat\Delta_\tau'$
be the projections of $\Delta'$ and $\Delta_\tau'$, respectively,
to the last $n-1$ coordinates.

Consider the part of $\Delta_\tau$ with $x_1\leq\alpha$.
This part can be written
\begin{align*}
\Big\{(x_1, x_2, \dots , x_n) \mathrel{}\big|\mathrel{}
         0 \leq x_1 \leq \alpha; \,\,
         \Big(\frac {x_2 \alpha} {x_1}, \dots, \frac {x_n \alpha} {x_1}\Big)
             \in \widehat\Delta'; \,\,
         x_n \geq \tau
     \Big\}.
\end{align*}
We introduce new coordinates $y_j := x_j \alpha / x_1$,
for $j \in \{2, \dots, n\}$.
In the coordinates $x_1, y_2, \dots, y_n$, the above set takes the form
\begin{align*}
\Delta_\tau\cap\{x_1\leq \alpha\}
   = \Big\{
         \frac{\tau\alpha}{y_n}\leq x_1\leq \alpha; \,\,
         y\in \widehat\Delta_\tau'
     \Big\}.
\end{align*}
Thus, we can write the contribution to $V(R)$ of this piece as
\begin{align*}
\frac {1} {k!}
      \int_{\widehat\Delta_\tau'}  & \int_{\tau \alpha / y_n}^\alpha
      \Big( \frac {x_1} {\alpha} \Big)^{n-1}
 \density^k \Big(x_1, \frac {x_1} {\alpha} y_2, \cdots, \frac {x_1} {\alpha} y_n \Big)
      \,\dee x_1 \,\dee\! \leb_{n-1} \\
 &=
\frac {1} {k!}
      \int_{\widehat\Delta_\tau'} \int_{\tau \alpha / y_n}^\alpha
         \frac {(\log y_n + \log x_1 - \log \alpha - \log \tau)^k} {x_1y_2 \cdots y_n}
      \,\dee x_1 \,\dee\! \leb_{n-1}.
\end{align*}
We do the inner integral using~(\ref{eqn:integrate_log_x_over_x}), and then
use the induction hypothesis:
\begin{align*}
& \frac {1} {(k+1)!}
      \int_{\widetilde\Delta_\tau'}
         \frac {(\log y_n + \log x_1 - \log \alpha - \log \tau)^{k+1}}
               {y_2 \cdots y_n} \biggr|_{\tau \alpha / y_n}^\alpha
      \,\dee\! \leb_{n-1} \\
& \qquad =
\frac {1} {(k+1)!}
      \int_{\widetilde\Delta_\tau'}
         \frac {(\log y_n - \log \tau)^{k+1}} {y_2 \cdots y_n}
      \,\dee\! \leb_{n-1} \\
& \qquad = \frac {R^{n+k}} {(n + k)!}
         + \frac {R^{{n+k}-1}} {({n+k}-1)!} \Big(
             \sum_{i, j\geq 2: \, i = j} \log v'_{ij}
         -   \sum_{i, j: \, i = j + 1} \log v'_{ij} \Big)
         + o(R^{{n+k}-1}).
\end{align*}
The sum of the logarithm terms within the parentheses can be re-expressed as
\begin{align*}
\log \Big(\frac {\alpha} {v_{2,1}} \Big)
   + \sum_{i, j\geq 2: \, i = j} \log v_{ij}
   - \sum_{i, j\geq 2: \, i = j + 1} \log v_{ij}.
\end{align*}
Here we have used that $v'_{ij} = \alpha v_{ij} / v_{i1}$,
for all $i$ and $j$ with $i\geq 2$.

Now we turn our attention to the part of $\Delta_\tau$
with $x_1\geq\alpha$.
Let $\widehat\Delta$ be the projection of $\Delta$ onto
the last $n-1$ coordinates. Choose small $\epsilon > 0$.
By Lemma~\ref{lem:only_x_axis_counts}, we may neglect any region outside
of the set 
\begin{align*}
K_\beta := \{x \in \R^n \mid \text{$x_j \leq \beta$, for $2\leq j \leq n-1$}\},
\end{align*}
where $\beta$ is an arbitrary positive number.
 The definition of $\Delta$ in~\eqref{eqn:Delta} implies that,
by taking $\beta$ small enough, we can ensure that
\begin{align*}
P_1 \subset
    K_\beta \intersection \Delta \intersection \{x\in\R^n \mid \alpha \leq x_1\}
  \subset P_2,
\end{align*}
where
\begin{align*}
P_1 &:=
   K_\beta \intersection \big([\alpha, v_{1,1} - \epsilon] \times \widehat\Delta\big)
\qquad\text{and} \\
P_2 &:=
   K_\beta \intersection \big([\alpha, v_{1,1} + \epsilon] \times \widehat\Delta\big).
\end{align*}
Denote by $P^\tau_1$, $P^\tau_2$, and $\widehat\Delta_\tau$ the intersection of
the set $\{x\in\R^n \mid x_n \geq \tau\}$ with, respectively,
$P_1$, $P_2$, and $\widehat\Delta$.

Let $K_\beta' := \{(x_j)_{j=2}^n\in\R^{n-1} \mid \text{$x_j \leq \beta$,
for $2\leq j\leq n-1$}\}$ be the projection of $K_\beta$ onto the last $n-1$
coordinates.

Using the induction hypothesis together with Lemma~\ref{lem:only_x_axis_counts},
\begin{align*}
\frac {1} {k!} \int_{P^\tau_1} \density^k(x) \,\dee\!\leb_n
    &= \Big( \int_{[\alpha, v_{1,1} - \epsilon]}
                 \frac {1} {x_1} \, \dee x_1 \Big)
       \Big( \frac {1} {k!}
          \int_{K_\beta' \intersection \widehat\Delta_\tau} \frac {(\log x_n - \log \tau)^k}
                                {x_2 \cdots x_n} \,\dee\!\leb_{n-1} \Big) \\
    &= \Big(\log(v_{1,1} - \epsilon) - \log \alpha \Big)
       \bigg(\frac {R^{n + k - 1}}
                   {(n + k - 1)!} + O\big(R^{n + k - 2}\big) \bigg).
\end{align*}
When the integral is over $P^\tau_2$ instead of $P^\tau_1$,
a similar equation holds with the sign of $\epsilon$ changed.
Since $\epsilon$ is arbitrarily small, we conclude that
\begin{align*}
\frac {1} {k!} \int_{
K_\beta \intersection \Delta_\tau \intersection \{x\in\R^n \mid \alpha \leq x_1\}
} \density^k(x) \,\dee\!\leb_n
    &= \big(\log v_{1,1} - \log \alpha \big)
       \frac {R^{n + k - 1}} {(n + k - 1)!} + o(R^{n + k - 1}).
\end{align*}
Combining this equation with the previous calculation, we get the result.
\end{proof}

The log terms in the statement of Lemma~\ref{lem:integral_induction}
follow the following pattern, shown here for $n=3$.
The plus and minus signs represent the sign of the log term in the expression;
a dot means that the term
is not present, and $0$ stands for a null coordinate.
\begin{align*}
  \begin{pmatrix}
      0		& 0      & 0	\\
      +		& 0      & 0	\\
      -		& +      & 0	\\
      \cdot	& -      & +	\\
  \end{pmatrix}.
\end{align*}
Note that when the vertices of $\Delta$ are of the following form,
\begin{align*}
  \begin{pmatrix}
      0 & 0      & 0		\\
      a & 0      & 0		\\
      a & b      & 0		\\
      a & b      & c		\\
  \end{pmatrix},
\end{align*}
the proof of the Lemma~\ref{lem:integral_induction} can be adapted
to give an exact formula for $V(R)$:
\begin{align*}
V(R) = \frac {(R + \log c)^{n+k}} {(n+k)!}.
\end{align*}

\subsection{Combining a flag simplex with the dual of a neighbour}

In this section, we take $\Delta$ to be of the same form as in
equation \eqref{eqn:Delta},
except with a single additional positive entry on the upper diagonal.
That is, the entry $v_{i,i+1}$ is now positive,
for some $i\in\{0, \dots, n - 1\}$.
As before, let $\newdelta_\tau$ be the set
$\newdelta \backslash (\R^{n-1} \times [0,\tau))$,
and, for each $R > 0$ and $k\in\N$, write
\begin{equation}
\label{eqn:v_one_def}
V_\Delta(R) := \frac {1} {k!}
        \int_{\newdelta_\tau}
        F^k(x)
        \,\dee\!\leb_n,
\end{equation}
where $\tau = \exp(-R)$. Sometimes we write simply $V(R)$ when the simplex is clear from context.

\begin{lemma}
\label{lem:integral_induction_neighbour}
If the extra positive entry is $v_{0,1}$, then, for all $k\in\N$,

\begin{align}
\label{eqn:two_coefficient_case}
V(R) = \frac {R^{{n+k}-1}} {({n+k}-1)!}
         \cdot \Bigl| \log \frac {v_{0,1}} {v_{1,1}} \Bigr|
         + o(R^{{n+k}-1}).
\end{align}
If, on the other hand, the extra positive entry is $v_{i,i+1}$,
with $i\in\{1, \dots, n - 1\}$, then
\begin{align}
\label{eqn:four_coefficient_case}
V(R) = \frac {R^{{n+k}-1}} {({n+k}-1)!}
         \cdot \Bigl| \log \frac {v_{i,i}} {v_{i, i+1}}
                     \frac {v_{i+1, i+1}} {v_{i+1, i}}         
           \Bigr|
         + o(R^{{n+k}-1}).
\end{align}
\end{lemma}

\begin{proof}
Again, we use induction on the dimension $n$.
When $n=1$, we can calculate the integral exactly
using~(\ref{eqn:integrate_log_x_over_x}),
and get that, for sufficiently small $\tau$,
\begin{equation*}
V(R) = \frac {1} {(k + 1)!}
              \Big|(\log v_{0,1} + R)^{k+1} - (\log v_{1,1} + R)^{k+1} \Big|.
\end{equation*}
This implies the statement of the lemma, in the case where $n=1$.

Now, assume that $n\geq 2$ and that the statement is true when the dimension
is $n - 1$. There are two possibilities to consider.

\emph{Case 1, where $v_{0,1} > 0$.}
For simplicity, we assume that $v_{0,1} > v_{1,1}$; the proof is similar
when the inequality is in the other direction.
The simplex $\newdelta$ can be expressed as the difference
(up to a set of measure zero)
of the two simplices
\begin{align*}
\tilde\Delta  &:= \conv\{0, v_0, v_2, \dots, v_n\} \quad \text{and} \\
\tilde\Delta' &:= \conv\{0, v_1, v_2, \dots, v_n\}.
\end{align*}
Both simplices are of the form assumed in Lemma~\ref{lem:integral_induction}.
Applying this lemma to each of them and subtracting the results
gives~(\ref{eqn:two_coefficient_case}).

\emph{Case 2, where $v_{i,i+1} > 0$ for some $i\geq 1$.}
As in the proof of Lemma~\ref{lem:integral_induction},
choose $\alpha >0$ such that $\alpha$ is less than each of the
first coordinates of the non-zero vertices of $\newdelta$.

Let $\Delta'$ be the intersection of the hyperplane $\{x_1 = \alpha\}$ with the
simplex $\newdelta$. Observe that the vertices of $\Delta'$ are exactly the
vertices of $\newdelta$ (apart from the origin) rescaled so as to have
the first coordinate equal to $\alpha$.
Denote these vertices by $v'_i$; $i\in\{1,\dots,n\}$.
Also, let $\Delta'_\tau$ be the intersection of $\{x_1 = \alpha\}$ with the
set $\newdelta_\tau := \newdelta \backslash \{x_n < \tau\}$,
which is not necessarily a simplex.

Look at the part $\newdelta_\alpha^-$ of $\newdelta_\tau$
on the same side of the hyperplane as the origin.
Just as in the proof of Lemma~\ref{lem:integral_induction}, we do a change
of coordinates and perform a one-dimensional integration to get that
the contribution of this part is

\begin{equation*}
\frac{1}{k!}\int_{\Delta_\alpha^-}F^k(x) \,\dee\!\leb_n=\frac {1} {(k+1)!}
      \int_{\Delta'_\tau}
         \frac {(\log x_n - \log \tau)^{k+1}} {x_2 \dots x_n}
      \,\dee\! \leb_{n-1}.
\end{equation*}
Now we use the induction hypothesis. If $v'_{1,2} > 0$, we get that
the contribution is
\begin{equation*}
\frac {R^{{n+k}-1}} {({n+k}-1)!}
         \cdot \Big| \log \frac {v'_{1,2}} {v'_{2,2}} \Big|
         + o(R^{{n+k}-1}).
\end{equation*}
Otherwise, we get
\begin{equation*}
\frac {R^{{n+k}-1}} {({n+k}-1)!}
         \cdot \Big| \log \frac {v'_{i,i}} {v'_{i, i+1}}
                     \frac {v'_{i+1, i+1}} {v'_{i+1, i}}
         \Big|
         + o(R^{{n+k}-1}),
\end{equation*}
where the extra positive entry is $v'_{i,i+1}$.
Recalling that $v'_{i,j} = \alpha v_{i,j} / v_{i,1}$, for each $i$ and $j$,
we get in both cases that the right-hand-side
of~(\ref{eqn:four_coefficient_case}) is a lower bound on $V(R)$.

Now we want to show that the same expression is also an upper bound.
Choose $\beta > 0$ such that $\beta$ is greater than each of the first
coordinates of the non-zero vertices of $\newdelta$.
Let $\Delta''$ be the projection of $\newdelta$ onto the hyperplane
$\{x_1 = \beta\}$ along rays passing through the origin.
In other words, $\Delta''$ is the intersection of that hyperplane
with the cone over $\Delta'$. Observe that the vertices of $\Delta''$
are exactly the points $(\beta / \alpha) v_i$, $i\in\{1, \dots, n\}$.
Another important thing to note is that $\conv(\{0\} \union \Delta'')$
contains $\newdelta$. So, we get an upper bound on $V(R)$ by calculating
the integral in (\ref{eqn:v_one_def}) with $\newdelta_\tau$ replaced by
$\conv(\{0\} \union \Delta'') \intersection \{x_n \geq \tau\}$.
This calculation is formally the same as the one we have just done, with
$\alpha$ replaced by $\beta$, and it has the same answer.
\end{proof}

Let $\Phi_\tau \colon \R^n \to \R^n$ be the map
$\Phi_\tau(x) := (1 - \tau) x + \tau v_n$
that shrinks by a factor $1 - \tau$ towards the point $v_n$.
Denote $\slopedDelta := \Phi_\tau(\Delta)$.
For each $R > 0$ and $k\in\N$, write
\begin{equation*}
\newV_\Delta(R) := \frac {1} {k!}
        \int_{\slopedDelta} \density^k(x)
        \,\dee\!\leb_n,
\end{equation*}
where $\tau = \exp(-R)$.
Again we omit the $\Delta$ when the simplex is clear from context.
Observe that if the $n$th coordinate of $v_n$ is $1$, and the $n$th
coordinate of every other $v_i$ is zero,
then $\tildedelta_\tau = \slopedDelta$ and so
$\tildeV_\Delta(R) = \newV_\Delta(R)$, for all $R>0$.

The case when the extra positive entry is the coordinate $v_{n-1, n}$
presents a special difficulty, because in this case the region we must integrate
over to compute the Funk volume is not $\tildedelta_\tau$ but instead $\slopedDelta$.
We show however that, for our purposes, the difference between the
two integrals is unimportant because it is of order $R^{n-2}$.
We first consider the two-dimensional case.

\begin{figure}
\input{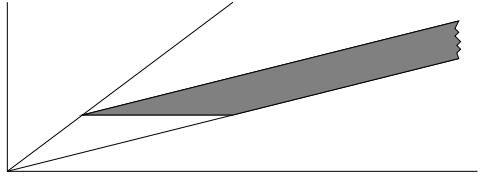}
\caption{The integral of $1/xy$ over the infinite grey region is finite. }
\label{fig:region_A}
\end{figure}

\begin{lemma}
\label{lem:awkward_case_2d}

Let $m_1$, $m_2$, and $\tau$ be positive real numbers, with $m_1 < m_2$.
Define the region
\begin{equation}
A_\tau := \Big\{(x, y) \subset \R^2
        \mid \text{$\tau \leq y$, and $m_1 x \leq y$,
                   and $m_1 (x - \tau/m_2) \geq y - \tau$} \Big\}.
\end{equation}
Then, the integral over $A_\tau$ of $1/xy$ against the Lebesgue measure is finite.
Moreover, it depends on $m_1$ and $m_2$, but not on $\tau$.
\end{lemma}

The region $A_\tau$ is illustrated in Figure~\ref{fig:region_A}.

\begin{proof}
Observe that the vertical distance between the two parallel lines bounding
$A_\tau$ is $h := \tau (1 - m_1 / m_2)$.
Since $y \geq m_1 x$ for $(x, y) \in A_\tau$, and the integrand is decreasing in $y$,
we have the following bound:
\begin{align*}
\int_{A_\tau} \frac {1} {xy} \, \dee y \, \dee x
    \leq \int_{\tau /m_2}^\infty \int_{m_1 x}^{m_1 x + h}
            \frac {1} {m_1 x^2} \, \dee y \, \dee x
    = \frac {-h} {m_1 x} \Big|_{\tau/m_2}^\infty
    = \frac {m_2} {m_1} - 1.
\end{align*}

That the integral is independent of $\tau$ follows from the invariance
of the integrand under the map $(x, y) \mapsto (\alpha x, \alpha y)$ mapping $A_\tau$ to $A_{\tau/\alpha}$, for any $\alpha > 0$.
\end{proof}

We now consider the $n$-dimensional case.
Note that $\slopedDelta \subset \tildedelta_\tau$ for all $\tau$.

\begin{lemma}
\label{lem:awkward_case}
Assume that $\Delta$ is of the form (\ref{eqn:Delta}),
except that the entry $v_{n-1, n}$ is positive.
Also, assume that $v_n = (1, \dots, 1)$. Then,
\begin{align*}
\diffV_\Delta(R)
     := \int_{\tildedelta_\tau \backslash \slopedDelta}
        \frac {\dee\!\leb_n} {x_1 \cdots x_n}
     = O(R^{n-2}).
\end{align*}
\end{lemma}

\begin{proof}
We treat the case where $v_{n-1, n-1} \geq v_{n-1, n}$.
The other case can be handled similarly.

Let $x\in \tildedelta_\tau \backslash \slopedDelta$.
Since $x\in \tildedelta_\tau$ we have

\begin{align*}
x_n &\geq \tau \\
\text{and}\qquad
\frac {x_n} {v_{n-1, n}} &\geq \frac {x_{n-1}} {v_{n-1, n-1}}.
\end{align*}

Since $x\in \Delta \setminus \slopedDelta$, we have
\begin{equation*}
\frac {x_n - \tau} {v_{n-1, n}} \leq \frac {x_{n-1} - \tau} {v_{n-1, n-1}}.
\end{equation*}
Let $W\subset\R^2$ be the set of pairs $(x_{n-1}, x_n)$ such that
these three inequalities hold.

Project $\Delta$ to the coordinate hyperplane $x_{n-1}=0$
to obtain the simplex $\overline\Delta$.
We may then apply Lemma~\ref{lem:quasi_standard} to $\overline\Delta$
and find positive constants $a_1,\dots, a_{n-2}, a_n$
such that, for all $x\in \Delta$, we have
\begin{equation*}
a_nx_n\leq a_{n-2}x_{n-2}\leq\dots\leq a_1x_1\leq 1.
\end{equation*}
In particular, for all $x\in\Delta_\tau$, we have
\begin{equation*}
\tau \leq \frac {a_{n-2}} {a_n} x_{n-2}
     \leq \frac {a_{n-3}} {a_n} x_{n-3}
     \leq \dots
     \leq \frac {a_1 x_1} {a_n}
     \leq \frac {1} {a_n}.
\end{equation*}
Let $\Delta'_\tau \subset \R^{n-2}$ be set of points
$(x_1,\dots,x_{n-2})$
satisfying this inequality. Observe that this set is a simplex.

Since $\tildedelta_\tau \backslash \slopedDelta$ is a subset of
$\Delta'_\tau \times W$, we have
\begin{equation*}
\diffV_\Delta(R) \leq
    \int_{\Delta'_\tau} \frac {\dee\!\leb_{n-2}} {x_1 \cdots x_{n-2}}
    \cdot\int_{W} \frac {\dee\!\leb_{2}} {x_{n-1} x_n}.
\end{equation*}
The first factor grows like $O(R^{n-2})$
according to Lemma~\ref{lem:integral_induction},
and the second is bounded, by Lemma~\ref{lem:awkward_case_2d}.
\end{proof}

\subsection{Combining a flag simplex with the dual of a non-neighbour}

Let the $n$-dimensional vectors $v_i$; $i\in\{0, \dots, n\}$ be such that
$v_{ij} > 0$, for all $j \leq i$.
Here $j$ indexes the coordinates and ranges from $1$ to $n$.
Assume that $\newnewdelta := \conv\{v_0, \dots, v_n\}$
is a non-degenerate $n$-dimensional simplex.

On $\R^n$ we use $\leq$ to denote the coordinate-wise partial order,
that is, $x \leq y$ if $x_j \leq y_j$ for all $j \in \{1, \dots, n\}$.
If $A\subset \mathbb R^n$, we write $|A|$ for its Lebesgue volume.

\begin{lemma}
\label{lem:make_smaller}
Let $\Delta_1$ and $\Delta_2$ be non-degenerate simplices in $\Rplus^n$
with vertices $u_i$; $i \in \{0, \dots, n\}$ and
$v_i$; $i \in \{0, \dots, n\}$ respectively.
Let $R>0$, and take $k=0$.
If $u_i \leq v_i$, for all $i\in\{0,\dots,n\}$, then
\begin{align*}
\frac {\newV_{\Delta_1}(R)} {|\Delta_1|}
    \geq \frac {\newV_{\Delta_2}(R)} {|\Delta_2|}.
\end{align*}
\end{lemma}

\begin{proof}
There exists an affine transformation $T$ taking $u_i$ to $v_i$
for all $i$. Each $x\in\Delta_1$ can be expressed as a convex combination
of the vertices $\{u_i\}$,
and so $Tx$ is a convex combination of the $\{v_i\}$ with the
same coefficients. We conclude that $Tx \geq x$, for all $x\in\Delta_1$.
Noting that $\density^0$ is monotone decreasing and that
$T$ takes $\tilde \Delta_1^\tau := \Phi_\tau(\Delta_1)$
to $\tilde \Delta_2^\tau := \Phi_\tau(\Delta_2)$, we get
\begin{align*}
\frac {|\tilde \Delta_1^\tau|} {|\tilde \Delta_2^\tau|}
\int_{\tilde \Delta_2^\tau} \density^0(x) \,\dee\!\leb_n
    &= \int_{\tilde \Delta_1^\tau} \density^0(Tx) \,\dee\!\leb_n
    \leq \int_{\tilde \Delta_1^\tau} \density^0(x) \,\dee\!\leb_n.
\end{align*}
It just remains to observe that $\Phi^\tau$ shrinks Lebesgue volume
by the constant factor $(1-\tau)^n$.
\end{proof}

\begin{lemma}
\label{lem:integral_induction_non_neighbour}
Let $\Delta$ be a non-degenerate simplex in $\Rplus^n$
with vertices $u_i$, $i \in \{0, \dots, n\}$, such that
$u_{i,j} >0$ when $i \geq j$.
Assume that there are two distinct vertices $v_{p-1}$
and $v_{q-1}$ of $\Delta$, with $p,q\in\{1,\dots,n\}$,
such that $v_{p-1,p}$ and $v_{q-1,q}$ are positive.
Take $k=0$. Then,
\begin{align*}
\newV_\Delta(R) = o(R^{n-1}).
\end{align*}
\end{lemma}

\begin{proof}
Let $\epsilon := \min\{v_{ij}\mid v_{ij}>0\}$ be the smallest strictly-positive coordinate
amongst the vertices of $\Delta$.
By Lemma~\ref{lem:make_smaller}, it is enough to show that the result is true
for the simplex obtained from $\Delta$ by replacing the lower triangular
coordinates, that is $v_{ij}$ with $i\geq j$, by $\epsilon$, the two diagonal coordinates
$v_{p-1,p}$ and $v_{q-1,q}$ by $\epsilon/2$, and all the other coordinates
by zero.
We verify that this simplex is non-degenerate by subtracting $v_0$ from
each of the other vectors, and calculating the determinant of the resulting
matrix, that is, $\det (v_{ij} - v_{0j})_{1\leq i,j \leq n}$.
This determinant equals $\epsilon^n/4$,
as can be easily verified by separately considering the cases
when $0\in\{q-1,p-1\}$ and $0\notin\{q-1,p-1\}$.

Thus we assume that $\Delta$ is of this new form.
Now scale $\Delta$ so that all coordinates of all vertices lie in $\{0,1/2, 1\}$. 
Notice that $\newV_\Delta$ does not change as the integrand is invariant
under diagonal linear maps.

Since we now have $v_{nn} = 1$,
the inclusion $\slopedDelta \subset \tildedelta_\tau$ holds,
for each $\tau\in(0,1)$, and so
\begin{align}
\label{eqn:sloped_is_smaller_than_flat}
\newV_\Delta(R) \leq \tildeV_\Delta(R),
\qquad\text{for all $R>0$}.
\end{align}

By switching labels if necessary, we can assume that $q>p$.
Let $v'$ be the vector having the same coordinates as $v_q$,
except for the $q$th coordinate, which we set equal to $0$.
So,
\begin{align*}
v'      &= (1, \dots, 1, 0,   0, \dots, 0) \\
v_{q-1} &= (1, \dots, 1, \nicefrac{1}{2}, 0, \dots, 0) \\
v_q     &= (1, \dots, 1, 1,   0, \dots, 0).
\end{align*}

Observe that the simplex $\Delta$ can be expressed as the difference
(up to a set of measure zero) of the two simplices
\begin{align*}
\Delta_1 &:= \conv\{v_0, \dots, v_{q-2}, v', v_{q-1}, v_{q+1}, \dots, v_n\}
\quad \text{and} \\
\Delta_2 &:= \conv\{v_0, \dots, v_{q-2}, v', v_q, v_{q+1}, \dots, v_n\}.
\end{align*}
Here $\Delta_1$ is the simplex obtained from $\Delta$ by replacing the
vertex $v_q$ with the vertex $v'$,
and $\Delta_2$ is obtained similarly by replacing $v_{q-1}$ with $v'$.
Both of these simplices are of the form assumed in
Lemma~\ref{lem:integral_induction_neighbour}.
We apply this lemma to each of them, with $k=0$, and see that the term
of order $n-1$ is the same in each case.
Subtracting, we get that $\tildeV_\Delta(R) = o(R^{n-1})$.
That the same holds for $\newV_\Delta(R)$ follows
from~(\ref{eqn:sloped_is_smaller_than_flat}).
\end{proof}

\section{The contribution of a flag simplex}
\label{sec:volume_contribution}

Recall that a flag decomposition of a polytope involves,
for each face $F$ of the polytope, a choice of point $p(F)$ in the relative
interior of $F$. The polytope $P$ is itself a face, and we choose $p(P)$
to be the origin, which we have assumed to be in the interior of $P$.
We also take a flag decomposition of the dual polytope,
and so for each face $F$ we also have a dual point $q(F)$
in the relative interior of the dual face of $F$.
To simplify notation, we write
\begin{align*}
L(F, G) := \log \Big(1 - \bigdotprod {q(F)} {p(G)}\Big),
\end{align*}
for faces $F$ and $G$ of the polytope $P$.
Observe that $L(F, G)$ is finite as long as $G$ is not a subset of $F$.
For a face $F$ of $P$, by $F^\circ$ we mean the face
of $\polytope^\circ$ that is dual to $F$.

We will make use of the following basic fact from linear algebra:
the space of linear anti-symmetric $2$-forms on $V$,
which can be thought of as the space  $(\land^2V)^*$ of linear functionals
on the exterior square of $V$, is naturally isomorphic to the exterior square
$\land^2 (V^*)$ of the dual.
The corresponding pairing $\land^2 V^*\times \land^2V\to \R$  is given on simple $2$-vectors by
\begin{equation*}
\langle v_1\land v_2, u_1\land u_2\rangle
    =\langle v_1, u_1\rangle \langle v_2, u_2\rangle
         - \langle v_1, u_2\rangle \langle v_2, u_1\rangle.
\end{equation*}

\begin{lemma}\label{lem:wedges_positive}
Fix $1\leq i\leq n-2$. Let $f$ be a flag of $P$, and $g:=\flip_i f$.
Choose points $p(f_j)$ in the relative interior of $f_j$,
and $q(g_j)$ in the relative interior of $g_j^\circ$,
for $j \in \{i-1,i,i+1\}$.
Denote
\begin{align*}
u_j &:= p(f_j)-p(f_{j-1}),
\qquad \text{for $j\in\{i,i+1\}$}; \\
v'_j &:= q(g_j)-q(g_{j+1}),
\qquad \text{for $j\in\{i, i-1\}$}.
\end{align*}
Then $\langle v'_i\land v'_{i-1}, u_i\land u_{i+1} \rangle\geq 0$.
\end{lemma}

\begin{proof}
Let $E$ be the $2$-dimensional affine plane through $q(g_j)$,
$j\in\{i-1, i, i+1\}$.
It lies in the affine span of $g_{i-1}^\circ=f_{i-1}^\circ$,
and since it contains $q(f_{i+1}^\circ)$, it intersects the face $f_i^\circ$
along an interval $I$. This interval cannot lie inside
$f_{i+1}^\circ=g_{i+1}^\circ$, for in this case $E$ would lie inside
the affine span of $g_i^\circ$, and so would not contain $q(f_{i-1}^\circ)$.
Therefore, $I$ passes through the relative interior of $f_i^\circ$.
Choose any point $q(f_i)$ in the interior of $I$,
and define $v_i := q(f_i)-q(g_{i+1})$ and $v_{i-1} := q(g_{i-1})-q(f_i)$.
Since $\{v_i, v_{i-1}\}$ and $\{v'_i, v'_{i-1}\}$
are bases of the linear space of $E$ having opposite orientations,
$v_i\land v_{i-1}=cv'_i\land v'_{i-1}$, for some $c<0$.
Thus it remains to check that
$\langle v_i\land v_{i-1}, u_i\land u_{i+1}\rangle \leq 0$.

Note that, since $\langle v_i, u_i\rangle =0$, we have
\begin{equation*}
\langle v_i\land v_{i-1}, u_i\land u_{i+1}\rangle
    = -\langle v_{i-1},u_i\rangle\langle v_i, u_{i+1}\rangle.
\end{equation*}
Now
\begin{equation*}
\langle v_{i-1}, u_i\rangle
    = \langle q(f_{i-1}) - q(f_i),  p(f_i)-p(f_{i-1})\rangle
    = \langle q(f_{i-1}),  p(f_i)\rangle -1\leq 0,
\end{equation*}
and similarly $\langle v_i, u_{i+1}\rangle\leq 0$. This concludes the proof.  
\end{proof}

\begin{lemma}
\label{lem:contribution_of_flag}
The volume of the Funk outward ball of radius $R$ inside a polytope $P$
satisfies
\begin{equation}
\label{eq:implicit_asymptotic}
\omega_n \volht_P\big(B_P(R)\big) = c_0(P) R^n + c_{1}(P) R^{n-1} + o(R^{n-1}),
\end{equation}
where $c_0(P) = |\flags(P)|/(n!)^2$, and
\begin{equation}
\label{eqn:single_flag}
\secondterm(P)
    = \frac {1} {n! (n-1)!}
      \sum_{f\in\flags(P)}
      \bigg(
          \sum_{i=0}^{n-1} L\big((\flip_i f)_{i},f_{i}\big)
        - \sum_{i=0}^{n-1} L\big((\flip_i f)_{i},f_{i+1}\big)
     \bigg).
\end{equation}
\end{lemma}

\begin{proof}
Lemma~\ref{lem:combine_all_the_flags} tells us that to find the contribution
of a flag $f$ to the volume, we must consider it paired
with each of the dual flag simplices.
We use the results of the previous section to determine the
contribution of each of these pairings.
This contribution depends on the degree of proximity of the dual flag to $f$.

Indeed, Lemma~\ref{lem:integral_induction_non_neighbour} shows that only
the dual of $f$ and its nearest neighbours are important,
when paired with $f$, while all other dual flags only contribute
to the $o(R^{n-1})$ error term. Recall that a nearest neighbour of a flag
is one that differs in exactly one face.
Moreover, by Lemma \ref{lem:integral_induction_neighbour},
only the pairing of the flag with its own dual contributes to the $R^n$ term.
Together, those observations confirm the asymptotic
formula~\eqref{eq:implicit_asymptotic},
and it remains to compute the values of $c_0(P)$ and $\secondterm(P)$.

We apply Lemma~\ref{lem:integral_induction} to the simplex
having vertices $v_i$, $0\leq i\leq n$ given by
\begin{equation*}
v_{ij} = x^f_{j-1}\big(p(f_i)\big) = 1 - \dotprod {q(f_{j-1})} {p(f_i)},
\qquad \text{for $1\leq j\leq n$},
\end{equation*}
taking $k=0$.
The contribution of $f$, paired with its dual,
to the leading term is $(1/n!) R^n$.
Together with Lemma \ref{lem:combine_all_the_flags},
the value for $c_0(P)$ follows.
The contribution to $\secondterm(P)$ of $f$ paired with its own dual is
\begin{equation*}
\sum_{i=0}^{n-1} L(f_{i},f_{i+1}) - \sum_{i=0}^{n-2} L(f_{i},f_{i+2}).
\end{equation*}

When we consider a nearest neighbour having a different face of dimension $i$,
all that changes is the $(i+1)$-th coordinate of each vertex.
Using Lemmas~\ref{lem:integral_induction_neighbour}
and~\ref{lem:awkward_case}, we get that the contribution to $\secondterm$
of $f$ paired with this neighbouring dual flag simplex is, up to sign,
\begin{equation*}
L\big((\flip_i f)_{i},f_{i}\big) + L(f_{i-1},f_{i+1})
    - L(f_{i-1},f_{i}) - L((\flip_i f)_{i},f_{i+1})
\end{equation*}
when $i\in\{1,\dots,n-1\}$,
while for $i=0$ we get
\begin{equation*}
L\big((\flip_0 f)_{0},f_{0}\big) - L((\flip_0 f)_0,f_{1})\text{.}
\end{equation*}

We will now confirm that we have chosen the sign correctly,
namely that these expressions are positive in all cases.
For $i=0$, this is straightforward to check, so we assume $i\geq 1$.

Denote $f' := \flip_if$. We must check that
\begin{equation*}
\big(1-\langle q(f_i'), p(f_i)\rangle\big)
        \big(1-\langle q(f_{i-1}), p(f_{i+1})\rangle\big)
    \geq 
        \big(1-\langle q(f_{i-1}), p(f_i)\rangle\big)
           \big(1-\langle q(f'_{i}), p(f_{i+1})\rangle\big).
\end{equation*}

First consider the case $i=n-1$.
Then $p(f_{i+1})=0$, and it remains to check that
\begin{equation*}
\langle q(f_{i-1}), p(f_i)\rangle \geq \langle q(f_i'), p(f_i)\rangle
    \iff \langle q(f_{i-1})-q(f_i'), p(f_i)\rangle\geq 0,
\end{equation*}
which holds since $q(f_{i-1})-q(f_i')=c (q(f_i)-q(f_{i-1}))$, for some $c>0$.

For $1\leq i\leq n-2$, set $u_j=p(f_{j})-p(f_{j-1})$ for $j\in\{i, i+1\}$,
and $v'_j=q(f'_j)-q(f'_{j+1})$ for $j\in  \{i-1, i\}$. One computes
\begin{align*}
\langle q(f_i'), p(f_i)\rangle
    &= 1 + \langle v'_i, u_i\rangle, \\
\langle q(f_{i-1}), p(f_{i+1})\rangle
    &= 1 + \langle v_i', u_i\rangle + \langle v_i', u_{i+1}\rangle
         + \langle v_{i-1}', u_i\rangle + \langle v_{i-1}', u_{i+1}\rangle, \\
\langle q(f_{i-1}), p(f_i)\rangle
    &= 1 + \langle v_i', u_i\rangle + \langle v_{i-1}', u_i\rangle, \\
\langle q(f_i'), p(f_{i+1})\rangle
    &= 1 + \langle v_i', u_i\rangle + \langle v_i', u_{i+1}\rangle\text{.}
\end{align*}
Then the desired inequality is easily seen to be equivalent to
\begin{equation*}
\langle v'_i, u_i\rangle \langle v'_{i-1}, u_{i+1}\rangle \geq \langle v'_{i-1},u_i\rangle \langle v'_{i},u_{i+1}\rangle\text{,}
\end{equation*}
which holds by Lemma \ref{lem:wedges_positive}.

Adding up all these contributions, and remembering that $L(f_{n-1},f_{n}) = 0$,
we get the stated formula for $\secondterm(P)$.
\end{proof}

This is illustrated in Figure~\ref{fig:contribution_of_flag}.
The dotted line represents the flag, connecting the full polytope face at
the top to the empty set at the bottom.
The arrows represent the terms of the above expression.
An arrow starting at a face $F$ and ending at face
$G$ represents the term $\pm\log\big(1 - \dotprod {q(G)} {p(F)}\big)$,
with a positive sign if the arrow is full, and a negative sign if the arrow
is dashed.

\begin{figure}
	\input{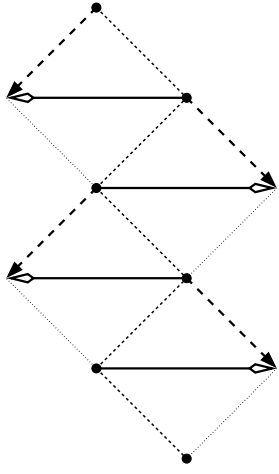}
	\caption{The contribution of a single flag simplex}
	\label{fig:contribution_of_flag}
\end{figure}

It will be convenient to express the sum
in~(\ref{eqn:single_flag}) in a different way.

\begin{lemma}
\label{lem:single_flag_rearranged}
Assume we are given a polytope $P$ with a flag decomposition,
and a flag decomposition of its dual. Then,
the sum over the flags
in~(\ref{eqn:single_flag}) is equal to
\begin{align*}
\sum_{f\in\flags(P)} \Big(
      \sum_{i=0}^{n-1} L\big((\flip f)_i, f_i\big)
    - \sum_{i=0}^{n-1} L\big((\flip f)_i, f_{i+1}\big) \Big).
\end{align*}
\end{lemma}

\begin{proof}
Fix $i$, and observe that $\flip_{i-1} \cdots \flip_0$ is a permutation
of the flags. If $g$ and $f$ are two flags such that
$g = \flip_{i-1} \cdots \flip_0 f$, then
\begin{align*}
g_i = f_i,
\qquad
g_{i+1} = f_{i+1},
\qquad \text{and} \quad
(\flip_i g)_i = (\flip f)_i.
\end{align*}
Therefore,
\begin{equation*}
\sum_{g\in\flags(\polytope)}
L\big((\flip_i g)_i, g_i \big)
= \sum_{f\in\flags(\polytope)}
L\big((\flip f)_i, f_i \big),
\end{equation*}
and
\begin{equation*}
\sum_{g\in\flags(\polytope)}
L\big((\flip_i g)_i, g_{i+1} \big)
= \sum_{f\in\flags(\polytope)}
L\big((\flip f)_i, f_{i+1} \big).
\end{equation*}
Summing over all $i$ gives the result.
\end{proof}

These terms are illustrated in Figure~\ref{fig:contribution_of_rearranged_flag}
(for a single flag $f$) using the same convention as before.

\begin{figure}
	\input{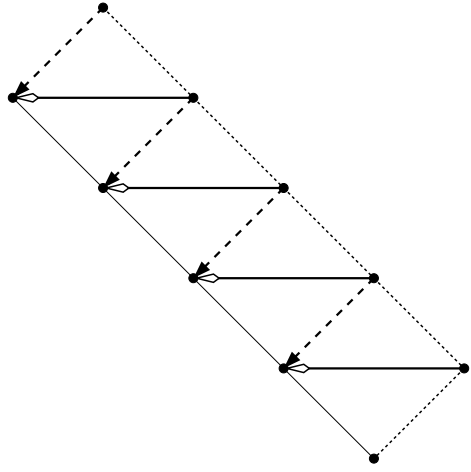}
	\caption{The contribution of a flag simplex after the sum is rearranged}
	\label{fig:contribution_of_rearranged_flag}
\end{figure}

\newcommand\aff{\operatorname{aff}}

In the following lemma, we are using the convention that faces
are closed. We use $\aff F$ to denote the affine hull of a set $F$,
and by $X - Y$ we mean the Minkowski sum of $X$ and $-Y$.

\begin{lemma}
\label{lem:equal_ratio}
Let $F$ and $G$ be faces of $\polytope$ such that $F\subset G$
and $\dim G = \dim F + 1$.
If $p, p' \in G$ and $q, q' \in F^\circ$,
then
\begin{align}
\label{eqn:equal_ratio}
\Big(1 - \dotprod {q} {p}\Big) \Big(1 - \dotprod {q'} {p'}\Big)
    = \Big(1 - \dotprod {q} {p'}\Big) \Big(1 - \dotprod {q'} {p}\Big).
\end{align}
\end{lemma}

\begin{proof}
Because the dimensions of $F$ and $G$ differ by exactly $1$,
we can write $p = s + \alpha v$ and $p' = s' + \alpha' v$,
with $s, s' \in \aff F$ and $\alpha, \alpha' \in \R$,
and where $v$ is some element of $(\aff G \setminus \aff F) - \aff F$.
Likewise, we can write $q = r + \beta u$ and $q' = r' + \beta' u$,
with $r, r' \in \aff G^\circ$ and $\beta, \beta' \in \R$,
and where $u$ is some element of
$(\aff F^\circ \setminus \aff G^\circ) - \aff G^\circ$.

Note that $\dotprod {r} {s} = 1$
and $\dotprod {r} {v} = \dotprod {u} {s} = 0$.
We calculate that $1 - \dotprod {q} {p} = - \alpha\beta \dotprod {u} {v}$,
and the other factors in~(\ref{eqn:equal_ratio}) can be calculated similarly.
The conclusion then follows.
\end{proof}

We are now ready to prove Theorem~\ref{thm:second_term}.

\begin{proof}[Proof of Theorem~\ref{thm:second_term}]
Let $f$ be a flag of $\polytope$.
For each $i\in\{0,\dots,n-2\}$, we apply Lemma~\ref{lem:equal_ratio}
with $F := (\flip f)_i$ and $G := f_{i+1}$.
Observe that $q((\flip f)_{i+1})$ and $q((\flip f)_{i})$ are in $F^\circ$,
and that $p(f_{i+1})$ and $p(f_0)$ are in $G$.
From the lemma,
\begin{align*}
L\big((\flip f)_{i+1}, f_0\big) + L\big((\flip f)_{i}, f_{i+1}\big)
    = L\big((\flip f)_{i+1}, f_{i+1}\big) + L\big((\flip f)_{i}, f_0\big).
\end{align*}
It is clear that no proper face of the flag $f$ is a subset of any proper face
of $\flip f$. Hence, all four of these terms are finite.
Summing over $i \in\{0,\dots,n-2\}$,
and using that $L\big((\flip f)_{n-1}, f_n\big)= 0$
(this is because $p(P)$ is the origin), we get
\begin{align*}
L\big((\flip f)_{n-1}, f_0\big)
+ \sum_{i=0}^{n-1} L\big((\flip f)_{i}, f_{i+1}\big)
= \sum_{i=0}^{n-1} L\big((\flip f)_{i}, f_{i}\big).
\end{align*}
Summing over all flags, and comparing with Lemmas~\ref{lem:contribution_of_flag}
and~\ref{lem:single_flag_rearranged}, we get the result.
\end{proof}

\begin{proof}[Proof of Corollary ~\ref{cor:flag_equality}]
Let $Q$ be a cube of the same dimension as $P$.
Both $P$ and $Q$ have $2^n n!$ flags.
Using Theorem \ref{thm:second_term} we may write
\begin{equation*}
\omega_n \volht_P \big(B_P(R) \big)
    = \frac{2^n}{n!} R^n
        + \frac{n}{(n!)^2}
          \sum_{f\in\flags(P)}
          \log \left(1 - \bigl\langle (\flip f)_{n-1}, {f_0}\bigr\rangle\right)
          R^{n-1}
        + o(R^{n-1}).
\end{equation*}
We have by Theorem~\ref{thm:unconditional_funk_mahler} that
\begin{equation*}
\volht_P \big(B_P(R) \big) \geq \volht_Q \big(B_Q(R) \big),
\qquad\text{for all $R>0$}.
\end{equation*}
As the leading term in the asymptotics for the volume
is the same for $Q$ and $P$, it must hold that
\begin{equation}
\label{eq:wrong_inequality}
\sum_{f\in\flags(P)}
        \log\left(1 - \bigl\langle (\flip f)_{n-1}, {f_0}\bigr\rangle\right)
    \geq
    \sum_{f\in\flags(Q)}
        \log\left(1 - \bigl\langle (\flip f)_{n-1}, {f_0}\bigr\rangle\right).
\end{equation}

For each $f\in\flags(Q)$, one easily verifies that
$\langle (rf)_{n-1}, f_0\rangle=-1$.
Thus the right hand side of \eqref{eq:wrong_inequality} equals $2^nn!\log 2$. 

For each $f\in\flags(P)$, it holds by the central symmetry of $P$
that $\langle (rf)_{n-1}, f_0\rangle\geq -1$,
and so none of the summands on the left-hand-side
of~\eqref{eq:wrong_inequality} exceed $\log2$.
As there are $2^nn!$ summands, it follows that each summand must equal
$\log 2$, that is, $\langle (rf)_{n-1}, f_0\rangle= -1$
for all $f\in \flags(P)$.
Thus $\langle  (rf)_{n-1}, -f_0\rangle= 1$, or equivalently,
$-f_0\in( rf)_{n-1}$, for all $f\in \flags(P)$, concluding the proof.
\end{proof}

\section{The limit of the Funk--Santal\'o point for polytopal Funk geometries
as the radius tends to infinity}
\label{sec:santalo}

Santal\'o~\cite{santalo_affine_invariant}
studied the dependence of the Mahler volume of a convex body
on the point in the body through which the polar is taken.
He showed that there is a unique point in the interior of the body where this
quantity is minimized. This point is now known as the \emph{Santal\'o point}
of the body.

This notion is generalised in~\cite{centro_affine_area_paper},
where, instead of the Mahler volume, the Holmes--Thompson volume of
metric balls of the Funk geometry was considered.
The following theorem appears there.
	
\begin{theorem}[\cite{centro_affine_area_paper}]
\label{thm:funk_santalo_existence}
Let $K$ be a convex body, and let $R>0$.
Then, the function $f_{K,R} \colon K \to \R$, defined by
\begin{align*}
f_{K, R}(x) := \volht_K\big(B_K(x, R)\big)
\end{align*}
is proper and strictly convex, and hence attains its infimum at a unique
point $s_R(K)$ in the interior of $K$.
\end{theorem}
	
Proper means that the function converges to infinity as the boundary of $K$
is approached.
We call the point $s_R(K)$ the \emph{Funk--Santal\'o point} of the
convex body $K$, for radius $R$.

It is shown in~\cite{centro_affine_area_paper}
that the Funk--Santal\'o point of a convex body converges to
the Santal\'o point as the radius tends to zero.
It is also shown that, in the case where the absolutely continuous part
of the surface area measure is not zero,
the Funk-Santal\'o point converges, as the radius tends to infinity,
to the point of the body with respect to which the centro-affine area
is minimized.
	
Here, we study the behaviour of the Funk-Santal\'o point as the radius becomes
large in the case of polytopes.
We will see that this behaviour is governed by the second highest order
term of the volume growth at infinity.

Recall that if a sequence $f_j$ of convex functions converges pointwise
on a dense subset $D$ of an open convex set $C$ in $\R^n$,
then the limit function $f$ exists on $C$,
and the convergence is locally uniform on $C$;
see~\cite[Thm.~10.8]{rockafellar_convex_analysis}.

The following lemma concerning the convergence of infima
of convex functions was established in~\cite{centro_affine_area_paper}.

\begin{lemma}
	\label{lem:convergence_of_minima}
	Let $f_j$ be a sequence of lower semicontinuous convex functions
	defined on a convex body $K$ in a finite-dimensional vector space,
	taking values in $\R \union \{\infty\}$.
	Assume that $f_j$ converges pointwise on $\interior K$ to a
	lower semicontinuous convex function $f$
	that is finite on $\interior K$.
	Then, $\inf f_j$ converges to $\inf f$.
	Moreover, if for each $j$, the function $f_j$ attains its infimum
	at $x_j$, and this sequence of points has a limit point $x$,
	then $f$ attains its infimum at $x$.
\end{lemma}

\begin{proof}[Proof of Theorem~\ref{thm:funk_santalo_polytope_convergence}]
By Theorem \ref{thm:second_term}
and Lemma \ref{lem:dual_with_respect_to_point}, we have
\begin{align}\label{eq:funk_santalo_function}
c_1(P, x)
    &= \frac {1} {n!(n-1)!} \sum_{f\in\flags(P)}
       \log \left(
           1 - \left\langle \frac {(\flip f)_{n-1}} {1 - \langle(\flip f)_{n-1}, x\rangle},f_0-x \right\rangle
            \right)
	\\
    &= C - \frac {1} {n!(n-1)!} \sum_{F\in F_{n-1}(P)}
		|\flags(F)| \log \big( 1 - \dotprod {q(F)} {x} \big), \nonumber
	\end{align}
where $C$ is a constant not depending on $x$,
and $F_{n-1}(P)$ denotes the set of facets of $P$,
and $q(F)$ is the vertex of $P^\circ$ corresponding to the facet $F$.

Notice that each non-constant term is minus the logarithm of an affine function,
and hence convex. In fact, the level sets of any one of these terms are
parallel hyperplanes, and the term is strictly convex along line segments
that do not lie in one of these hyperplanes. Since the vertices of $P^\circ$,
that is, the set of dual vectors $q(F)$ with $F\in F_{n-1}(P)$,
affinely generate the whole dual space,  these points separate $\R^n$,
and so there is no non-degenerate line segment lying in a level set of
every one of the terms. We conclude that $c_1(P, x)$ is strictly convex.
	
The function $c_1(P,x)$ is clearly continuous, and finite in $\interior P$.
If $x$ is on the boundary of $P$, then there is at least one facet $F$
for which $x \in  F$, or equivalently $\dotprod {q(F)} {x} = 1$.
Since none of the terms can be $-\infty$, we deduce that $c_1(P, x) = \infty$.
This establishes that $c_1(P, x)$ is proper.

It follows immediately that the infimum of $c_1(P, x)$ is attained
at a unique point, which we denote $s_\infty(P)$.

As $c_1(P,x)$ is smooth in $x$, $s_\infty(P)$ must be its unique stationary point. Thus
\[\sum _{F\in F_{n-1}(P)}|\flags(F)|\frac{1}{1-\langle q(F), s_\infty(P)\rangle}q(F)=0,  \]
and assuming $s_\infty(P)=0$ we obtain the asserted  characterization of $s_\infty(P)$.

For each $R>0$, define a function of $x\in \interior P$ by
\begin{equation*}
f_{P,R}(x)
    := \frac {1} {R^{n-1}}
    \Big(\omega_n \volht\big(B_P(x, R)\big)
           - \frac {1} {(n!)^2}|\flags(P)|R^n \Big).
\end{equation*}
By Theorem~\ref{thm:funk_santalo_existence}, this function is a proper
convex function, for all $R>0$, and its infimum is attained at $s_R(P)$.
By Theorem \ref{thm:second_term} it holds that
\begin{equation*}
\lim_{R\to\infty} f_{P, R}(x) = c_1(P, x),
\qquad\text{for all $x\in \interior P$}.
\end{equation*}
	Thus the conditions of Lemma~\ref{lem:convergence_of_minima} are satisfied,
	and it follows upon applying that lemma that $s_R(P)$ converges to
	$s_\infty(P)$, as $R\to\infty$.
\end{proof}

\section{Stationary points of the 2nd-highest-order-term functional in 2-d}
\label{sec:stationary}

In this section, we prove Theorem \ref{thm:regular_stationary},
that is, that, in dimension two, the only stationary points
of the second highest order term $c_1(P)$ in the expansion
of the volume of balls are the regular polytopes, up to linear transformations. We show moreover that they uniquely maximize $c_1(P, s_\infty(P))$.

The proof will rely on the following lemma concerning polygons where
$\secondterm$ remains stationary under certain perturbations of a single vertex.
Recall that $\polygons^m$ denotes the space of convex polygons in $\R^2$
with $m$ vertices that contain the origin in their interior.

\begin{lemma}
\label{lem:flipping_map}
Let $v_1$, $v_2$, $v_3$ be consecutive vertices of a polygon $K\in \polygons^m$.
Denote by $e_2$, $e_3$ and $e_4$ the consecutive vertices of $K^\circ$
such that
\begin{align*}
\dotprod {e_2} {v_1} = \dotprod {e_2} {v_2} = 1 = \dotprod {e_3} {v_2}
= \dotprod {e_3} {v_3} = \dotprod {e_4} {v_3}.
\end{align*}
For small $t\in\R$, let $K^t$ be the polygon obtained from $K$ by moving $v_2$
by an amount $t(v_2 - v_1)$. If $t$ is small enough,
then $K^t$ is a convex polygon. Assume that
\begin{align*}
\frac {d} {dt} \secondterm(K^t) \Big|_{t=0} = 0.
\end{align*}
Then, there is an element $\gamma$ of $\gltwo$ such that
$\gamma(v_1) = e_4$, $\gamma(v_2) = e_3$, and $\gamma(v_3) = e_2$.
\end{lemma}

\begin{proof}
As $t$ varies, so do $v_2$ and $e_3$.
So, there are four terms in the expression for $\secondterm(K^t)$ that change
with $t$. More precisely,
\begin{align*}
\secondterm(K^t)
    &= \log\big(1 - \dotprod {e_1} {v_2(t)}\big)
    + \log\big(1 - \dotprod {e_4} {v_2(t)}\big) \\
    & \quad\quad\quad {} + \log\big(1 - \dotprod {e_3(t)} {v_1}\big)
    + \log\big(1 - \dotprod {e_3(t)} {v_4}\big) + C,
\end{align*}
where $C$ is a constant independent of $t$.
Here, $v_4$ and $e_1$ are the vertices of $K$ and $K^\circ$, respectively,
such that $v_2$, $v_3$, and $v_4$ are consecutive,
and $e_1$, $e_2$, and $e_3$ are consecutive.
So,
\begin{equation*}
\frac {d} {dt} \secondterm(K^t) \Big|_{t=0}
    = -\frac {\dotprod {e_1} {v_2 - v_1}} {1 - \dotprod {e_1} {v_2}}
     - \frac {\dotprod {e_4} {v_2 - v_1}} {1 - \dotprod {e_4} {v_2}}
     - \frac {\dotprod {\dot e_3} {v_1}} {1 - \dotprod {e_3} {v_1}}
     - \frac {\dotprod {\dot e_3} {v_4}} {1 - \dotprod {e_3} {v_4}}.
\end{equation*}
Here, $\dot e_3$ denotes the derivative of $e_3$ with respect to $t$,
at $t=0$.

To simplify calculations, we choose $(e_3, e_2)$ as a coordinate basis
for $\R^2$. With respect to this basis
\begin{align*}
v_1 = (x, 1), \qquad
v_2 = (1, 1), \qquad \text{and} \quad
v_3 = (1, y),
\end{align*}
with $x$ and $y$ in $(-\infty, 1)$.
Moreover, with respect to the dual basis,
\begin{align*}
e_2 = (0, 1), \qquad
e_3 = (1, 0), \qquad \text{and} \quad
e_4 = (w, z),
\end{align*}
where $w$ and $z$ satisfy
\begin{align}
\label{eqn:goes_through_v3}
w + y z = 1.
\end{align}
Since $\dotprod {e_4} {v_3 - v_2} >0$, we have $z<0$.
A short calculation shows that
\begin{align*}
\dot e_3 = \frac {1-x} {1-y} (y, -1).
\end{align*}

Denote the coordinates of $v_4$ by $v_4 = (a, b)$.
Since $\dotprod {e_4} {v_3} = \dotprod {e_4} {v_4} = 1$,
we have $w + y z = a w + b z=1$, and hence
\begin{align*}
\frac {b - y} {1 - a} = \frac {w} {z}.
\end{align*}
Using these coordinates, and that $\dotprod {e_1} {v_1} = 1$, we get
\begin{equation*}
\frac {d} {dt} \secondterm(K^t) \Big|_{t=0}
    = 1
     - \frac {(1-x)w} {1 - w - z}
     + \frac {1 - xy} {1-y}
     - \frac {1-x} {1-y} \cdot \frac {a y - b} {1 - a}.
\end{equation*}
Observing that
\begin{equation*}
-\frac {a y - b} {1 - a}
    = \frac {y (1 - a) + (b - y)} {1 - a}
    = y + \frac {w} {z},
\end{equation*}
we have
\begin{align*}
\frac {d} {dt} \secondterm(K^t) \Big|_{t=0}
    &= 1
     + \frac {(1-x)w} {(1 - y)z}
     + \frac {1 - xy} {1-y}
     + \frac {1-x} {1-y} \Big(y + \frac {w} {z}\Big) \\
    &= \frac {2} {1-y} \Big((1-x) \frac {w} {z} + 1 - xy \Big).
\end{align*}
By assumption, this expression is zero, so we obtain a linear relation between
$w$ and $z$, which combined with~(\ref{eqn:goes_through_v3}) gives us a pair
of equations having the unique solution
\begin{align*}
w = \frac {1 - xy} {1 - y} \qquad\text{and}\quad
z = -\frac {1 - x} {1 - y}.
\end{align*}
The linear map defined by the matrix
\begin{align*}
\gamma := \frac {1} {1 - y} \begin{pmatrix} -y & 1 \\ 1 & -1 \end{pmatrix}
\end{align*}
satisfies $\gamma(v_1) = e_4$, $\gamma(v_2) = e_3$, and $\gamma(v_3) = e_2$.
\end{proof}

We will also need to know that the Funk-Santal\'o point depends smoothly on $P$.

\begin{lemma}\label{lem:smooth_santalo}
The map $s_\infty \colon \polygons^m\to\R^2$ is $C^\infty$-smooth.
\end{lemma}

\begin{proof}
As $s_\infty$ commutes with translations, it suffices to consider
the open subset $U$ of polygons containing $0$ in their interior.
For $P\in U$, let $e_i$, $1\leq i\leq m$ be a cyclical ordering
of the vertices of $P^\circ$.
By Theorem~\ref{thm:funk_santalo_polytope_convergence}
and  Equation~\eqref{eq:funk_santalo_function},
$s_\infty(P)$ is the unique minimum of the following strictly-convex
smooth function defined on the interior of $P$:
\begin{equation*}
f_P(z)=-\sum_{i=1}^m\log \big( 1-\langle e_i, z\rangle \big).
\end{equation*}
It follows that $s_\infty(P)$ is the unique solution of the equation
$H(P, \cdot) = 0$, where
\begin{equation*}
H(P, z) := \nabla f_P(z) = \sum_{i=1}^m \frac{1}{1-\langle e_i, z\rangle}e_i.
\end{equation*}
Let us verify that $\partial H / \partial z$ has rank $2$,
so that by the implicit function theorem $z=s_\infty(P)$
depends smoothly on $P$.

We write $z = (z_1, z_2)$ and $e_i = (e_i^1, e_i^2)$, for each $i$.
Observe that the Jacobian is
\begin{equation*}
\frac{\partial H_j}{\partial z_k}
    = \sum_{i=1}^m \frac {e_i^j e_i^k} {(1-\langle e_i, z\rangle)^2},
\qquad\text{for $j\in\{1, 2\}$}.
\end{equation*}
For clarity we write $x(e_i), y(e_i)$ for the coordinates of $e_i$.
So, the determinant is given by
\begin{align*}
\det \frac{\partial H}{\partial z}
    &= \sum_{i=1}^m \frac {x(e_i)^2} {(1-\langle e_i, z\rangle)^2}
       \sum_{j=1}^m \frac {y(e_j)^2} {(1-\langle e_j, z\rangle)^2}
     -\left(
       \sum_{i=1}^m\frac{x(e_i)y(e_i)}{(1-\langle e_i, z\rangle)^2}
      \right)^2 \\
    &= \sum_{1\leq i<j\leq m}
       \frac {\big(x(e_i)y(e_j)-x(e_j)y(e_i) \big)^2}
             {\big(1-\langle e_i, z\rangle \big)^2
              \big(1-\langle e_j, z\rangle \big)^2}.
\end{align*}
If $\partial H / \partial z$ were not of rank $2$,
then we would have $x(e_i)y(e_j)=x(e_j)y(e_i)$, for all $i$ and $j$,
that is, all the points $e_i$ would lie on a line, which is impossible.
\end{proof}

\begin{lemma}\label{lem:global_max}
The function $f(P)=c_1(P, s_\infty(P))$ has a global maximum on $\polygons^m$, for each $m\geq 3$.
\end{lemma}
\begin{proof}
	
Let $X_m$ denote the set of $m$-sided convex polygons $P\subset\R^2$
in John's position. Recall that $P$ is said to be in John's position if the John ellipse of $P$---the unique maximal
volume ellipse it contains---is the unit Euclidean ball centered at the origin.

Since $f(P)$ is invariant under affine transformations of $P$,
and any $P\in\polygons^m$ has an affine image in John's position,
it suffices to show that $f(P)$ attains a global maximum on $X_m$.

Let $B$ denote the Euclidean unit disc around the origin.
It follows from John's theorem that each polygon $P\in X_m$
is contained in $2 B$. Also, as $P$ contains $B$, it holds that
the polar body $P^\circ$ lies inside $B$.
We conclude that, for all $v\in P$ and $e\in P^\circ$,
one has $|\langle e, v\rangle|\leq 2$.
It follows that each summand in the expression
in~\eqref{equ:second_term_dim_two} for $c_1(P)=c_1(P, 0)$ is bounded
from above by $(1/2)\log 3$.

Let $v_1,\dots, v_m$ by the consecutive vertices of $P$,
with cyclical index, and let $e_1,\dots, e_m$ be the vertices of $P^\circ$
such that $\langle e_i, v_i\rangle=\langle e_i, v_{i-1}\rangle=1$, for all $i$.

By Blaschke's selection theorem, $X_m$ is precompact within the space
$\mathcal K(\R^2)$ of convex bodies equipped with the Hausdorff distance.
Its closure $\overline X_m$ in $\mathcal{K} (\R^2)$
is $X_m \union \partial X_m$,
where $\partial X_m := X_{m-1} \union \dots \union X_3$
is the geometric boundary of $X_m$.

Suppose now that $P \in X_m$ approaches the boundary $\partial X_m$.
This implies that there is some $i$ for which either $|v_i - v_{i+1}|$
or $|e_i - e_{i+1}|$ tends to zero.
If the former is the case, then
\begin{equation*}
1-\langle e_{i}, v_{i+1}\rangle = \langle e_{i}, v_i-v_{i+1}\rangle
\end{equation*}
tends to zero.
Similarly, in the latter case, $1 - \langle e_{i+1}, v_{i-1}\rangle$
tends to zero.

We conclude that at least one of the summands
in~\eqref{equ:second_term_dim_two} for $c_1(P)$ tends to $-\infty$.
Since we have seen that each of the summands is bounded above,
we conclude that $c_1(P)$ tends to $-\infty$,
as $P$ approaches the boundary.
It follows that $f(P)$ also tends to the same limit as $P\to\partial X_m$, 
since
\begin{equation*}
f(P)=c_1(P, s_\infty(P))\leq c_1(P,0)=c_1(P),
\end{equation*}
by the definition of $s_\infty(P)$.

Using this, and that the function $f(P)$ is continuous on $X_m$,
we deduce that it attains a global maximum in $X_m$,
and thus in $\mathcal P^m$, concluding the proof.
\end{proof}

\begin{proof}[Proof of Theorem~\ref{thm:regular_stationary}]
Let the function $\secondterm$ be stationary at a polygon $K \in \polygons^m$.
Let $v_1, \dots, v_m$ be the consecutive vertices of $K$
with cyclical index, and $e_1,\dots, e_m$ the consecutive vertices of $K^\circ$
as in Lemma~\ref{lem:flipping_map},
so that $\langle e_i, v_i\rangle=\langle e_i, v_{i-1}\rangle=1$,
for all $i$.
By that lemma, there is an element $\gamma$
of $\gltwo$ such that the triple of vertices
$(v_1, v_2, v_3)$ is mapped to the triple $(e_4, e_3, e_2)$ of dual vertices.
We apply the same lemma again, this time to $K^\circ$, to get that there
exists another element $\gamma'$ of $\gltwo$ such that
the triple $(e_2, e_3, e_4)$ is mapped to $(v_4, v_3, v_2)$.
Composing these maps, we get an element $\gamma''$ of $\gltwo$ taking the
triple of points $(v_1, v_2, v_3)$ to $(v_2, v_3, v_4)$.
Similarly, we have an element of $\gltwo$ taking the triple $(v_2, v_3, v_4)$
to $(v_3, v_4, v_5)$. Note that this is actually the same map,
because it agrees with $\gamma''$ on $v_2$ and $v_3$.
Proceeding inductively, we get that $\gamma''$ takes $v_j$ to $v_{j+1}$,
for all $j$.
It follows that $(\gamma'')^m$ is the identity map.
The iterates $(\gamma'')^k$; $k\in\{0, \dots, m-1\}$ form a finite subgroup
of $\gltwo$,
and are therefore contained in some maximal compact subgroup conjugate to
the orthogonal group $\orthtwo$. Thus, we may write $\gamma''=g h g^{-1}$,
for some $g\in \gltwo$ and $h\in \orthtwo$.
So, $K = g \regular_m$, where $\regular_m$ is the regular polygon
$\regular_m := \conv\big\{h^k (g^{-1}v_1)
: \text{$k \in \{0, \dots, m-1\}$}\big\}$.

It remains to show that $P=\regular_m$ maximizes $c_1(P, s_\infty(P))$
among all $P\in\polygons^m$ with $s_\infty(P)$ at the origin,
uniquely up to linear transformations.

By Lemma \ref{lem:smooth_santalo}, $s_\infty \colon \polygons^m\to\R^2$
is smooth. As it commutes with translations, it is everywhere submersive.
It follows that $Y_m:=\{P\in\polygons^m : s_\infty(P)=0\}$
is a smooth submanifold of $\polygons^m$. 

By Lemma \ref{lem:global_max}, the function $c_1(P, s_\infty(P))$
has a global maximum on $\polygons^m$.
In fact, since it is affinely invariant,
the maximum, which we denote by $P_0$, can be taken to be in $Y_m$.
Note that on $Y_m$, we have $c_1(P, s_\infty(P))=c_1(P)$.

We can write the tangent space at $P_0$ as
\begin{equation*}
T_{P_0}\polygons^m=T_{P_0}Y_m\oplus T_{P_0}(G P_0),
\end{equation*}
where $G=\R^2$ is the group of translations of $\R^2$.

It follows from the definition of $s_\infty$ that $P_0$ is a local minimum
of $c_1(P)$ in the directions of $T_{P_0}(G P_0)$,
while by the above it is a local maximum in the directions of $T_{P_0}Y_m$.
Thus, $P_0$ is a stationary point of $c_1(P)$ in $\polygons^m$.
By the first part of the theorem,
$P_0$ coincides with $\regular_m$ up to a linear transformation.
The assertion of the theorem readily follows.
\end{proof}

Next we consider the derivative of $V_K(\lambda) := \omega_n \volht_P(B_K(R))$,
where the radius $R$ is related to $\lambda$ by $R = -\log(1-\lambda)$.
In general, this derivative can be written as an integral over the boundary
of the ball.
For $K=P\in \polygons^m$, this integral can be computed explicitly, as follows.

Let $v_1,v_2,\dots, v_m$ be the consecutive vertices of  $P$
with cyclical index, and $e_1,\dots, e_m$ the consecutive vertices
of $P^\circ$, so that $\langle e_i, v_i\rangle=\langle e_i, v_{i-1}\rangle=1$,
for all $i$. For each $i$ and $j$,
write $\oneminusprod_{ij} := 1 - \lambda \dotprod {e_j} {v_j}$,
understood as a function of $\lambda$.

\begin{proposition}
It holds that
\begin{align*}
\frac {dV_K(\lambda)} {d\lambda}
    &= \sum_{i,j} \lambda\cdot 
           \Bigl| \log \frac
		{ \oneminusprod_{i,j} \, \oneminusprod_{i+1, j+1} }
		{ \oneminusprod_{i,j+1} \, \oneminusprod_{i+1, j} }
		\Bigr| \cdot \frac
		{|\dotprod {e_i} {v_j} \dotprod {e_{i+1}} {v_{j+1}}
			- \dotprod {e_i} {v_{j+1}} \dotprod {e_{i+1}} {v_j} |}
		{| \oneminusprod_{i,j} \, \oneminusprod_{i+1,j+1}
			- \oneminusprod_{i,j+1} \, \oneminusprod_{i+1,j} |} 
\\& \!\!\!\!\!\!\!\!\!\!\!\!\!\!\!\! = \sum_{i,j}
    \Bigl| \log \frac
        { \oneminusprod_{i,j} \, \oneminusprod_{i+1, j+1} }
        { \oneminusprod_{i,j+1} \, \oneminusprod_{i+1, j} }
    \Bigr|
    \cdot \left|\lambda -
          \frac {\langle e_i, v_j\rangle +\langle e_{i+1}, v_{j+1}\rangle
              -\langle e_{i}, v_{j+1}\rangle -\langle e_{i+1}, v_{j}\rangle}
                {\dotprod {e_i} {v_j} \dotprod {e_{i+1}} {v_{j+1}}
              - \dotprod {e_i} {v_{j+1}} \dotprod {e_{i+1}} {v_j}  }
          \right|^{-1}.
\end{align*}
\end{proposition}

\begin{proof}
Observe that the proof of Lemma~\ref{lem:combine_all_the_flags}
can be repeated to show that
\begin{equation*}
\omega_n \volht(B_P(0, R))
        = \frac {1} {2} \sum_{i=1}^m \sum_{j=1}^m
              \int_{\lambda \Delta(0, v_i, v_{i+1})}
              \frac {\dee y_j \dee y_{j+1}} {y_j y_{j+1} },
\end{equation*}
where $y_j = 1 - \langle e_j, \cdot \rangle$,
and $\Delta(0, v_i, v_{i+1})$ is the triangle with the listed vertices.
We then parametrize $p = s(1-t) v_i + stv_{i+1}$,
$0\leq s\leq \lambda$, $0\leq t\leq 1$,
and use the fundamental theorem of calculus to deduce the stated formula.
\end{proof}

Let us finish this section by pointing out that for the regular polygon
$\regular_m$ we have a very simple formula for second order term, as follows:
\begin{align*}
c_1(\regular_m)
    &= 2m \log \Bigl(2\sin \Bigl(\frac \pi m \Bigr)\Bigr) \\
    &= 2m \log \Bigl(\frac{l}{r}\Bigr),
\end{align*}
where $l$ is the length of an edge,
and $r$ is the radius of the circumscribed circle.

\section{Simplices}\label{sec:simplices}

\begin{theorem}
Consider the Funk geometry of the simplex $\Delta_n\subset \R^n$.
The Holmes--Thompson volume of the forward ball $\Delta_n(R)$ of radius $R$ centered
at the barycenter is $1/\omega_n$ times
\begin{equation}
  V_{\Delta_n}(R)=\frac{n+1}{n!} \int_{\Delta_n(R)}\frac{\text{d}x_1\cdots\text{d}x_n}{x_1\cdots x_n},
\end{equation} 
We also have a recursive formula for $n\geq2$,
\begin{equation}
  V'_{\Delta_n}(R)=\frac{n+1}{n}\frac{1}{1-\frac{e^{-R}}{n+1}} V_{\Delta_{n-1}}\biggl(R+\log\Bigl(1+\frac1n-\frac1n e^{-R}\Bigr)\biggr).
\end{equation}
For all $n\geq 1$ it holds for $R\to\infty$ that
\begin{equation}
   V_{\Delta_n}(R)=\frac{(n+1)!}{(n!)^2}R^n\bigl(1+n\log(n+1)R^{-1}+o(R^{-1})\bigr),
\end{equation}
while for $R\to 0$ we have 
\begin{equation}
 V_{\Delta_n}(R)= \frac{(n+1)^{n+1}}{(n!)^2} \Bigl(R^n -\frac{n}{2} R^{n+1} +o(R^{n+1})\Bigr).
\end{equation} 
\end{theorem}

\begin{proof}

We may consider the standard $n$-dimensional simplex $\Delta_n$, with vertices $C_0=0, C_1=e_1,\ldots,C_n=e_n$, where $e_i$ is the standard basis.
We will consider a point $C=\sum_{i=0}^n\alpha_iC_i\in\Delta_n$, with $\alpha_i\geq 0$ and $\sum\alpha_i=1$, whose
barycentric coordinate are therefore $(\alpha_0,\ldots,\alpha_n)$.
Then in standard Euclidean coordinates,
$$
C=(\alpha_1,\ldots,\alpha_n)
$$
Notice that $\alpha_0=0$ corresponds to the face $F_0$ belonging to the
hyperplane $\sum_{i=1}^{n}x_i=1$, and for $1\leq i\leq n$,  $\alpha_i=0$ corresponds
to the face $F_i$ belonging to the hyperplane $x_i=0$.
It follows that the vectors
$$
V_1=(-\alpha_1,0,\ldots,0), V_2=(0,-\alpha_2,0,\ldots,0),\ldots, V_n=(0,\ldots,0,-\alpha_n) 
$$
are unit vector with respect to the Funk metric at $C$.
The same is true for the vector
$$
V_0(\frac{1-\sum_{i=1}^n\alpha_i}n,\ldots,\frac{1-\sum_{i=1}^n\alpha_i}n).
$$
hence for $i=1,\ldots,n$ the hyperplanes $X_i=-\alpha_i$ are the faces of the tangent
unit ball, and so is the hyperplane $\sum_{i=1}^nX_i=1-\sum_{i=1}^n\alpha_i=\alpha_0$.
Then the dual simplex can be seen as the convex hull
of the following vectors in the dual space: 
$$
W_1=(\frac{-1}{\alpha_1},0,\ldots,0),\ldots,W_n=(0,\ldots,\frac{-1}{\alpha_n})
$$
and 
$$
W_0=\frac{1}{\alpha_0}(1,\ldots,1).
$$
The volume of this simplex is the sum of the volumes of the simplices defined
by $0$ and any $n$ of this points.
For instance, the volume defined by $0, W_1,\ldots,W_n$
is exactly
$$
\frac{1}{n!}\frac{1}{\alpha_1\cdots\alpha_n}
$$
hence by adding all these terms
we get the  Holmes--Thompson density in the following form
\begin{equation}\label{densityofnsimplex}
  \beta_{\Delta_n}^*(A)=\frac{1}{n!} \sum_{i=0}^n\frac{1}{\alpha_0\cdots\hat\alpha_i\cdots\alpha_n}=\frac{1}{n!}\frac{1}{\alpha_0\cdots\alpha_n}
\end{equation}

(where the $\hat{}$ means that the term is removed).
Recall that the Funk ball $\Delta_n(R)$ is the image of $\Delta_n$ under the dilation with ratio $(1-e^{-R})$
centered at the barycenter $\frac1{n+1}(1,\ldots,1)$,
then
$$
V_{\Delta_n}(R)=\frac{1}{n!} \sum_{i=0}^n \int_{\Delta_n(R)}\frac{1}{\alpha_0\cdots\hat\alpha_i\cdots\alpha_n} \text{d}x
$$
If for $i=1,\ldots,n$ we use the change of variable, 
$$F_i(x)=(x_1,\ldots,x_{i-1}, 1-\sum_{k=1}^nx_k,x_{i+1},\ldots , x_n)$$
we get that
$$
\int_{\Delta_n(R)}\frac{1}{\alpha_0\cdots\hat\alpha_i\cdots\alpha_n} \text{d}x = \int_{\Delta_n(R)}\frac{1}{\alpha_1\cdots\alpha_n} \text{d}x
$$
which allows us to get the following formula for the volume of the forward ball
of radius $R$ centered at the barycenter of the simplex
\begin{equation}
  \label{eq:Funkvolumesimplex1}
  V_{\Delta_n}(R)=\frac{n+1}{n!} \int_{\Delta_n(R)}\frac{\text{d}x_1\cdots\text{d}x_n}{x_1\cdots x_n}
\end{equation}
in euclidean coordinates.

Denote by \[A_0=\bigl(\frac1{n+1},\ldots,\frac1{n+1}\bigr)\] the barycenter of $\Delta_n$.
Then for $1\leq k\leq n$,
\[ A_k(R)=A_0+ (1-e^{-R})\left(\frac{1}{n-k}\sum_{j=1}^{n-k}e_j-A_0\right)\]
are the vertices of a simplex $S_n(R)$ in the barycentric subdivision of $\Delta_n(R)$. 

Explicitly,
\begin{eqnarray}
 A_{1}(R)&=&\bigl(\frac1n-\frac{e^{-R}}{n(n+1)},\ldots, \frac1n-\frac{e^{-R}}{n(n+1)},\frac{e^{-R}}{n+1}\bigr)\\ 
\nonumber &\ldots& \\ 
  A_k(R)&=& \bigl(\underbrace{\frac1{n-k+1}-\frac{k}{(n+1)(n-k+1)}e^{-R}}_{n-k \text{ times}} ,\underbrace{\frac{e^{-R}}{n+1}}_{k \text{ times}}\bigr)\\
\nonumber &\ldots& \\ 
 A_n(R)&=&\bigl(\frac{e^{-R}}{n+1},\ldots,\frac{e^{-R}}{n+1}\bigr) \nonumber
\end{eqnarray}

Therefore, because of the symmetries involved, using the right equality in
 equation~(\ref{densityofnsimplex}) with $X=x_1+\cdots+x_n$, we have
\begin{equation}
  \label{eq:decompositionBarycentricFunkSimplex}
  \int_{\Delta_n(R)}\frac{\text{d}x_1\cdots\text{d}x_n}{(1-X)x_1\cdots x_n} = (n+1)! \int_{S_n(R)}\frac{\text{d}x_1\cdots\text{d}x_n}{(1-X)x_1\cdots x_n}
\end{equation}

Let us focus on 

\begin{equation}
	\label{eq:deuxiememethode}
	   W_n(R)=\int_{S_n(R)}\frac{\text{d}x_1\cdots\text{d}x_n}{(1-x_1-\cdots-x_n)x_1\cdots x_n},
		\end{equation}
so that 
\[
V_{\Delta_n}(R)=(n+1) W_n(R)\text{.}
\]
We can rewrite $W_n(R)$ using Fubini's theorem and denoting $s=x_n$
\begin{equation*}
  \begin{split}
   W_n(R)&= \int_{\frac{e^{-R}}{n+1}}^{\frac{1}{n+1}} \frac{\text{d}s}{s}\int_{S_{n}(-\log(s(n+1)))\cap\{x_n=s\}} \frac{\text{d}x_1\cdots\text{d}x_{n-1}}{\bigl((1-s)-x_1-\cdots-x_{n-1}\bigr)x_1\cdots x_{n-1}}.
\end{split}
\end{equation*}
Now let us proceed with the change of variable $t_i=x_i/(1-s)$ for $i=1,\ldots,n-1$  in the inner integral. This gives
\begin{equation}
  \begin{split}
    W_n(R)&=\int_{\frac{e^{-R}}{n+1}}^{\frac{1}{n+1}}\frac{\text{d}s}{s(1-s)}\\
 &\int_{S_{n-1}\left(-\log\frac{sn}{1-s}\right)} \frac{\text{d}t_1\cdots\text{d}t_{n-1}}{(1-t_1-\cdots-x_{n-1})\cdot t_1\cdots t_{n-1}} \\
& = \int_{\frac{e^{-R}}{n+1}}^{\frac{1}{n+1}}\frac{\text{d}s}{s(1-s)} W_{n-1}\left(-\log\frac{s\cdot n}{1-s}\right) 
  \end{split}
\end{equation}
Hence the following holds
$$
W'_n(R)=\frac{1}{1-\frac{e^{-R}}{n+1}}W_{n-1}\biggl(-\log \frac{n}{(n+1)e^R-1}\biggr),
$$
from which we get the following differential equation by multiplying by $(n+1)$:
\begin{equation}
  \label{eq:volumesimplexefunk}
  V'_{\Delta_n}(R)=\frac{n+1}{n}\frac{1}{1-\frac{e^{-R}}{n+1}} V_{\Delta_{n-1}}\biggl(R+\log\Bigl(1+\frac1n-\frac1n e^{-R}\Bigr)\biggr),
\end{equation}
as claimed.

Next consider the asymptotic behavior as $R\to 0$. The one dimensional case follows from Lemma \ref{lem:same_hanner}, yielding
\[
V_{\Delta_1}(R)=2 \log(2e^R-1)=2R + 2\log 2 + 2\log(1-\frac{e^{-R}}2)=4R-2R^2+o(R^2).
\]

Next we proceed by induction, utilizing the recursive formula. Writing 
$$
V_{\Delta_n}(R)=a_nR^n+b_nR^{n+1}+o(R^{n+1}),
$$
and taking into account that 
$$
R+\log\Bigl(1+\frac1{n+1}-\frac1{n+1} e^{-R}\Bigr) =\frac{n+2}{n+1}R\bigl(1-R/2(n+1)+o(R)\bigr),
$$
we get that
\begin{equation}
  \label{eq:Inductionsimplexe1}
  \begin{split}
   W_{\Delta_n}(R):=  V_{\Delta_{n}}& \biggl(R+\log\Bigl(1+\frac1{n+1}-\frac1{n+1} e^{-R}\Bigr)\biggr)\\
&=\left(\frac{n+2}{n+1}\right)^na_nR^n + \left(\frac{n+2}{n+1}\right)^n\Bigl(\frac{n+2}{n+1}b_n - a_n\frac{n}{2(n+1)}\Bigr) R^{n+1}\\ &+o(R^{n+1}).
  \end{split}
\end{equation}
Now eq. \eqref{eq:volumesimplexefunk} implies
$$
V'_{\Delta_{n+1}}(R)= \left(\frac{n+2}{n+1}\right)^2\biggl(1-\frac{R}{n+1}+o(R)\biggr)W_{\Delta_n}(R)
$$
to obtain
\begin{multline*}
  V'_{\Delta_{n+1}}(R)= \Bigl(\frac{n+2}{n+1}\Bigr)^{n+2}a_nR^n\biggl(1-\frac{R}{n+1}+o(R)\biggr)\\
  \times \biggl(1 +\frac1{a_n}\Bigl(\frac{n+2}{n+1}b_n - a_n\frac{n}{2(n+1)}\Bigr)R +o(R) \biggr),
\end{multline*}
which finally gives us
\begin{equation}
  \label{eq:inductionsimplexe2}
   V'_{\Delta_{n+1}}(R)=\Bigl(\frac{n+2}{n+1}\Bigr)^{n+2}a_nR^n\biggl(1 +\frac{n+2}{a_n(n+1)}\Bigl(b_n - \frac12a_n\Bigr)R +o(R)\biggr).
\end{equation}
After integration we get
$$
a_{n+1} = \Bigl(\frac{n+2}{n+1}\Bigr)^{n+2} \frac{1}{n+1}a_n
$$
and
$$
b_{n+1}= \Bigl(\frac{n+2}{n+1}\Bigr)^{n+2} \frac{1}{(n+1)}\Bigl(b_n - \frac12a_{n}\Bigr),
$$
that is 
$$
b_{n+1}= \Bigl(\frac{n+2}{n+1}\Bigr)^{n+2} \frac{1}{(n+1)}b_n - \frac12a_{n+1} .
$$
Taking into account that $a_1=4$, those equations readily imply for all $n\geq 2$
we easily deduce by induction that for all $n\geq 2$, \[ a_n=(n+1)^{n+1}/(n!)^2,\qquad
b_{n}=-\frac{n}{2} \frac{(n+1)^{n+1}}{((n)!)^2},\]
as stated.

Finally, the asymptotics as $R\to\infty$ follow at once from Theorem \ref{thm:second_term}, or they can be deduced from the recursive formula as in the $R\to 0$ regime.
\end{proof}

\section{Some further reflections and questions}
\label{sec:questions}

Another conjecture related to flags and Mahler volume,
and which may be related to Funk volume, is the following.

\begin{conjecture}
\label{conj:kalai_volume_growth}
Let $P$ be a centrally symmetric polytope in dimension $n$.
Then,
\begin{equation*}
|\flags(P)| \geq \frac {(n!)^2} {2^n} |P| |P^\circ|.
\end{equation*}
\end{conjecture}

Kalai talks about this conjecture in a YouTube video,
around 32 minutes in; see~\cite{kalai_video}.
He ascribes it to conversations with Freij, Henze, Schmitt, Ziegler,
about a decade ago, based on computer experimentation.
We might be so bold as to make the following conjecture.

\begin{conjecture}
\label{conj:our_volume_growth}
Let $\Omega$ be a centrally symmetric convex body,
and let $H$ be a Hanner polytope of the same dimension.
Then, for any $r$ and $R$ positive real numbers with $r<R$,
\begin{equation*}
\frac {{\volht_\Omega} \big(B_\Omega(R)\big)}
      {{\volht_\Omega} \big(B_\Omega(r)\big)}
    \geq \frac {{\volht_H} \big(B_H(R)\big)} {{\volht_H} \big(B_H(r)\big)}.
\end{equation*}
\end{conjecture}

This is a strengthening of Conjecture~\ref{conj:kalai_volume_growth}.
Indeed, we recover that conjecture from Conjecture \ref{conj:our_volume_growth}
by letting $r$ approach zero, and $R$ approach infinity.

Conjecture~\ref{conj:our_volume_growth} reminds us of the Bishop--Gromov
theorem in Riemannian geometry, but with the inequality reversed. The inequality in Funk geometry directly corresponding to the Bishop-Gromov theorem would be \begin{equation*}
\frac {{\volht_\Omega} \big(B_\Omega(R)\big)}
      {{\volht_\Omega} \big(B_\Omega(r)\big)}
    \leq \frac {{\volht_E} \big(B_E(R)\big)} {{\volht_E} \big(B_E(r)\big)},
\end{equation*}
where $E$ is a centered ellipsoid. This inequality is false:
considering small radii, it would imply the inequality
\begin{equation*}
\int_{\Omega\times \Omega^\circ} \langle x,\xi\rangle^2 \,\dee x \,\dee\xi
    \leq \frac{n}{(n+2)^2} |\Omega| |\Omega^\circ|,
\end{equation*}
which was shown to be false for large $n$
by Klartag~\cite[Proposition 1.5]{klartag_mahler},
even for unconditional bodies.

The paper~\cite{freij_henze_schmitt_ziegler_face_numbers_of_centrally_symmetric_polytopes_produced_from_split_graphs}
studies Hansen polytopes of split graphs. This class contains
polytopes that are only slightly worse than the Hanner polytopes with
respect to the Mahler conjecture, the flag conjecture, and also Kalai's
$3^d$ conjecture concerning the minimal number of faces
of a centrally-symmetric polytope.
It would be interesting to see how closely these polytopes come to minimizing
the volume of Funk balls.

In ~\cite{schutt_polytopes_floating},
it was shown that the number of flags of a polytope determines
the asymptotics of the volume of its floating body. This phenomenon was later studied in greater generality in ~\cite{besau_schutt_werner_flag_numbers_and_floating_bodies}.
Since for smooth bodies the asymptotics is governed by the affine surface
area, the authors consider the flag number to be an analogue of the
affine surface area for polytopes. One could ask if it is possible to
obtain a refinement that depends on the geometry of the polytope,
similarly to how we have the second-highest order term for the Funk volume?

It was conjectured in~\cite{faifman_funk}, that, if one considers
the volume of a ball in a Funk geometry centered about the centroid
of the body, then this quantity is maximised when the body is an ellipsoid.
This is an upper bound counterpart to Conjecture~\ref{conj:main_conjecture}.
It likewise interpolates between two inequalities, this time that of
Blaschke--Santal\'o when the radius goes to zero,
and the Centro-affine Isoperimetric inequality when the radius
tends to infinity; see~\cite{centro_affine_area_paper}.
This upper bound conjecture was established~\cite{faifman_funk}
in the case of unconditional bodies.

\bibliographystyle{plain}
\bibliography{polytopes.bib}

\begin{bibdiv}
\begin{biblist}

\bib{alexander_fradelizi_zvavitch}{article}{
      author={Alexander, Matthew},
      author={Fradelizi, Matthieu},
      author={Zvavitch, Artem},
       title={Polytopes of maximal volume product},
        date={2019},
        ISSN={0179-5376},
     journal={Discrete Comput. Geom.},
      volume={62},
      number={3},
       pages={583\ndash 600},
         url={https://doi.org/10.1007/s00454-019-00072-3},
      review={\MR{3996937}},
}

\bib{besau_schutt_werner_flag_numbers_and_floating_bodies}{article}{
      author={Besau, Florian},
      author={Sch{\"u}tt, Carsten},
      author={Werner, Elisabeth~M.},
       title={Flag numbers and floating bodies},
        date={2018},
        ISSN={0001-8708},
     journal={Adv. Math.},
      volume={338},
       pages={912\ndash 952},
}

\bib{chambers}{article}{
      author={Chambers, Greg},
       title={A note on {Kalai}'s $3^d$ {C}onjecture},
        date={2022},
      eprint={arXiv:2211.09215v1},
}

\bib{faifman_funk}{article}{
      author={Faifman, Dmitry},
       title={A {Funk} perspective on billiards, projective geometry and
  {Mahler} volume},
     journal={J. Differential Geom.},
      eprint={arXiv:2012.12159},
        note={To appear},
}

\bib{centro_affine_area_paper}{misc}{
      author={Faifman, Dmitry},
      author={Vernicos, Constantin},
      author={Walsh, Cormac},
       title={Blashke metric and the centro-projective volume of convex
  bodies},
        note={In preparation},
}

\bib{fradelizi_meyer}{article}{
      author={Fradelizi, Matthieu},
      author={Meyer, Mathieu},
       title={Increasing functions and inverse {S}antal\'{o} inequality for
  unconditional functions},
        date={2008},
        ISSN={1385-1292},
     journal={Positivity},
      volume={12},
      number={3},
       pages={407\ndash 420},
         url={https://doi.org/10.1007/s11117-007-2145-z},
      review={\MR{2421143}},
}

\bib{freij_henze_schmitt_ziegler_face_numbers_of_centrally_symmetric_polytopes_produced_from_split_graphs}{article}{
      author={Freij, Ragnar},
      author={Henze, Matthias},
      author={Schmitt, Moritz~W.},
      author={Ziegler, G{\"u}nter~M.},
       title={Face numbers of centrally symmetric polytopes produced from split
  graphs},
        date={2013},
        ISSN={1077-8926},
     journal={Electron. J. Comb.},
      volume={20},
      number={2},
       pages={research paper p32, 15},
         url={www.combinatorics.org/ojs/index.php/eljc/article/view/v20i2p32},
}

\bib{kalai_the_number_of_face_of_centrally_symmetric_polytopes}{article}{
      author={Kalai, Gil},
       title={The number of faces of centrally-symmetric polytopes},
        date={1989},
        ISSN={0911-0119},
     journal={Graphs Comb.},
      volume={5},
      number={4},
       pages={389\ndash 391},
}

\bib{kalai_video}{misc}{
      author={Kalai, Gil},
       title={Challenges in discrete geometry --- part {I}: Convex polytopes},
   publisher={YouTube. {\tt https://youtu.be/gnVluHk9N6Y}},
        date={2021},
        note={Uploaded by Fields Institute},
}

\bib{klartag_mahler}{article}{
      author={Klartag, Bo'az},
       title={Isotropic constants and {M}ahler volumes},
        date={2018},
        ISSN={0001-8708},
     journal={Adv. Math.},
      volume={330},
       pages={74\ndash 108},
         url={https://doi.org/10.1016/j.aim.2018.03.009},
      review={\MR{3787541}},
}

\bib{meyer_unconditional}{article}{
      author={Meyer, Mathieu},
       title={Une caract\'{e}risation volumique de certains espaces norm\'{e}s
  de dimension finie},
        date={1986},
        ISSN={0021-2172},
     journal={Israel J. Math.},
      volume={55},
      number={3},
       pages={317\ndash 326},
         url={https://doi.org/10.1007/BF02765029},
      review={\MR{876398}},
}

\bib{meyer_reisner_mahler_polygons}{article}{
      author={Meyer, Mathieu},
      author={Reisner, Shlomo},
       title={On the volume product of polygons},
        date={2011},
        ISSN={0025-5858},
     journal={Abh. Math. Semin. Univ. Hambg.},
      volume={81},
      number={1},
       pages={93\ndash 100},
         url={https://doi.org/10.1007/s12188-011-0054-3},
      review={\MR{2812036}},
}

\bib{reisner_unconditional}{article}{
      author={Reisner, Shlomo},
       title={Minimal volume-product in {B}anach spaces with a
  {$1$}-unconditional basis},
        date={1987},
        ISSN={0024-6107},
     journal={J. London Math. Soc. (2)},
      volume={36},
      number={1},
       pages={126\ndash 136},
         url={https://doi.org/10.1112/jlms/s2-36.1.126},
      review={\MR{897680}},
}

\bib{rockafellar_convex_analysis}{misc}{
      author={{Rockafellar}, R.~Tyrrell},
       title={{Convex analysis}},
   publisher={{Princeton, N. J.: Princeton University Press}},
        date={1970},
}

\bib{santalo_affine_invariant}{article}{
      author={{Santal\'o}, L.~A.},
       title={{Un invariante afin para los cuerpos convexos del espacio de $n$
  dimensiones}},
    language={Spanish},
        date={1949},
        ISSN={0032-5155},
     journal={{Port. Math.}},
      volume={8},
       pages={155\ndash 161},
        note={(Spanish)},
}

\bib{sanyal_werner_ziegler_on_kalais_conjectures}{article}{
      author={Sanyal, Raman},
      author={Werner, Axel},
      author={Ziegler, G{\"u}nter~M.},
       title={On {Kalai}'s conjectures concerning centrally symmetric
  polytopes},
        date={2009},
        ISSN={0179-5376},
     journal={Discrete Comput. Geom.},
      volume={41},
      number={2},
       pages={183\ndash 198},
}

\bib{schutt_polytopes_floating}{article}{
      author={Sch\"{u}tt, Carsten},
       title={The convex floating body and polyhedral approximation},
        date={1991},
        ISSN={0021-2172},
     journal={Israel J. Math.},
      volume={73},
      number={1},
       pages={65\ndash 77},
         url={https://doi.org/10.1007/BF02773425},
      review={\MR{1119928}},
}

\bib{shaping_space_book}{book}{
      editor={Senechal, Marjorie},
       title={Shaping space. {Exploring} polyhedra in nature, art, and the
  geometrical imagination. {With} {George} {Fleck} and {Stan} {Sherer}},
   publisher={New York, NY: Springer},
        date={2013},
        ISBN={978-0-387-92713-8; 978-0-387-92714-5},
}

\bib{vernicos_walsh_flag_approximability_of_convex_bodies}{article}{
      author={Vernicos, Constantin},
      author={Walsh, Cormac},
       title={Flag-approximability of convex bodies and volume growth of
  {Hilbert} geometries},
        date={2021},
        ISSN={0012-9593},
     journal={Ann. Sci. {\'E}c. Norm. Sup{\'e}r. (4)},
      volume={54},
      number={5},
       pages={1297\ndash 1314},
}

\end{biblist}
\end{bibdiv}

\end{document}